\documentclass[a4paper,11pt]{article}
\usepackage[margin=2.5cm]{geometry}

\usepackage{amssymb}

\usepackage[utf8]{inputenc}
\usepackage{cite}
\usepackage{amsmath}
\usepackage{amssymb}
\usepackage{amsthm}
\usepackage{faktor} 
\usepackage{hyperref}
\usepackage{bm}
\usepackage{subcaption} 
\usepackage{enumitem} 
\usepackage{algpseudocode}
\usepackage[ruled, lined, longend, linesnumbered]{algorithm2e}
\usepackage{microtype}
\usepackage[normalem]{ulem} 
\usepackage{pifont}
\newcommand{\cmark}{\ding{51}}%
\newcommand{\xmark}{\ding{55}}%
\usepackage{booktabs}
\usepackage{overpic}

\usepackage{pgfplots}
\pgfplotsset{compat=newest}
\usepackage{tikz}

\DeclareMathOperator*{\argmin}{\mathrm{arg\,min}}
\DeclareMathOperator{\tr}{tr}

\newcommand{\llangle}{\left\langle}
\newcommand{\rrangle}{\right\rangle}

\renewcommand{\vec}{\bm}
\newcommand{\mat}{\bm}
\newcommand{\tensor}{\bm}

\newcommand{\displacement}{\vec{u}_\mathrm{s}}
\newcommand{\bardisplacement}{\bar{\vec{u}}_\mathrm{s}}

\newcommand{\permeability}{\bm{k}_\mathrm{f}}

\newcommand{\rhos}{\rho_\mathrm{s}}
\newcommand{\rhof}{\rho_\mathrm{f}}

\newcommand{\fluidvelocity}{\bm{v}_\mathrm{f}}
\newcommand{\barfluidvelocity}{\bar{\bm{v}}_\mathrm{f}}
\newcommand{\porosity}{\phi_0}

\newcommand{\stress}{\tensor{\sigma}}

\newcommand{\g}{\vec{g}}
\newcommand{\f}{\vec{f}}

\newcommand{\n}{\vec{n}}

\newcommand{\NABLA}{\vec{\nabla}}
\newcommand{\DIV}{\NABLA\cdot}
\newcommand{\GRAD}{\NABLA}
\newcommand{\eps}[1]{\tensor{\varepsilon}(#1)}

\newcommand{\discreteenergy}{\mathcal{J}}

\def\spaceU{\bm{V}}
\def\spaceP{Q}
\def\spaceV{\bm{W}}

\def\testspaceU{\bm{V}}
\def\testspaceP{Q}
\def\testspaceV{\bm{W}}

\def\spaceUh{\bm{V}_h}
\def\spacePh{Q_h}
\def\spaceVh{\bm{W}_h}

\newcommand{\testu}{\bm{u}^\star}
\newcommand{\testp}{p^\star}
\newcommand{\testv}{\bm{v}^\star}

\def\FlowDirichletBoundary{\Gamma^D}

\def\MechanicsDirichletBoundary{\Gamma^D}

\newcommand{\incp}[1]{d_p^{#1}}
\newcommand{\incu}[1]{\bm{d}_{\bm{u}}^{#1}}
\newcommand{\incv}[1]{\bm{d}_{\bm{v}}^{#1}}

\newcommand{\solidstabilization}{\bm{\beta_\mathrm{s}}}
\newcommand{\fluidstabilization}{\bm{\beta}_\mathrm{f}}
\newcommand{\pressurestabilization}{\beta_\mathrm{p}}

\newcommand{\hatsolidstabilization}{\bm{\widehat{\beta}_\mathrm{s}}}
\newcommand{\hatpressurestabilization}{\widehat{\beta}_\mathrm{p}}

\newcommand{\weightedsquare}[2]{\llangle\!\llangle {#1} \rrangle\!\rrangle_{#2}^2}

\def\us{\bm{u}_\mathrm{s}}
\def\vs{\bm{v}_\mathrm{s}}
\def\vf{\bm{v}_\mathrm{f}}
\def\vv{\vs,\vf,m}
\def\dvv{\delta\vs,\delta\vf,\delta m}
\def\ovv{\overline{\vs},\overline{\vf},\overline{m}}

\def\Solid{\mathcal{S}}
\def\Fluid{\mathcal{F}}
\def\Mass{\mathcal{M}}

\def\Piola{\mathbf{P}}
\def\ff{\mathbf{f}}
\def\tt{\mathbf{t}}
\def\nn{\mathbf{n}}
\def\FF{\mathbf{F}}
\def\EE{\mathbf{E}}
\def\II{\mathbf{I}}
\def\Cauchy{\tensor\sigma}
\def\Cauchyf{\tensor\sigma_\mathrm{f}}
\def\Cauchys{\tensor\sigma_\mathrm{s}}
\def\Frech{\mathcal{D}}

\newtheorem{theorem}{Theorem}
\newtheorem{assumption}{Assumption}
\newtheorem{lemma}[theorem]{Lemma}

\newtheorem{remark}[theorem]{Remark}
\numberwithin{equation}{section}
\numberwithin{theorem}{section}

\newcommand*\samethanks[1][\value{footnote}]{\footnotemark[#1]}

\author{Jakub W.\ Both\thanks{The first two authors have equally contributed to the work.} \thanks{PMG, Department of Mathematics, University of Bergen, Norway}
\and Nicolas A.\ Barnafi\samethanks[1] \thanks{MOX, Department of Mathematics, Politecnico di Milano, Italy}
\and Florin A.\ Radu \samethanks[2]
\and Paolo Zunino \samethanks[3]
\and Alfio Quarteroni \samethanks[3]}

 \title{Iterative splitting schemes for a \\soft material poromechanics model}
 
\date{}

\begin{document}
\maketitle


 
\begin{abstract}
We address numerical solvers for a poromechanics model particularly adapted for soft materials, as it generally respects thermodynamics principles and energy balance. Considering the multi-physics nature of the problem, which involves solid and fluid species, interacting on the basis of mass balance and momentum conservation, we decide to adopt a solution strategy of the discrete problem based on iterative splitting schemes. As the model is similar (but not equivalent to) the Biot poromechanics problem, we follow the abundant literature for solvers of the latter equations, developing two approaches that resemble the well known undrained and fixed-stress splits for the Biot model. A thorough convergence analysis of the proposed schemes is performed. In particular, the undrained-like split is developed and analyzed in the framework of generalized gradient flows, whereas the fixed-stress-like split is understood as block-diagonal $L^2$-type stabilization and analyzed by means of a relative stability analysis. In addition, the application of Anderson acceleration is suggested, improving the robustness of the split schemes. Finally, we test these methods on different benchmark tests, and we also compare their performance with respect to a monolithic approach. Together with the theoretical analysis, the numerical examples provide guidelines to appropriately choose what split scheme shall be used to address realistic applications of the soft material poromechanics model.

\vspace{0.2cm}


\end{abstract}

%







\section{Introduction \label{section:introduction}}

Poromechanics addresses the behavior of fluid-saturated permeable porous materials, and in particular the interaction of their mechanical deformation and the fluid flow. Since its origin in the context of civil engineering \cite{Biot1941155,Biot1955182,Biot195691,Terzaghi1943}, most commonly known as Biot's theory of poroelasticity, it has been used for countless applications of societal and industrial relevance, e.g., in reservoir geomechanics, hydrology and soil mechanics, and material sciences (see the review \cite{deBoer2005} and the references therein). More recently, it has also captured the attention of researchers interested in the behavior of highly deformable, soft biological tissues \cite{Yang1991587,Taber1996323,Goriely2015931}; a prominent example is the perfusion of the heart \cite{Nash2000113,Chabiniok2016,Chapelle201091,di2020computational}.

As the classical theory of poroelasticity and resulting models were originally developed for civil applications, they are in general inadequate for biomechanics. This difficulty can be more evidently appreciated when considering soft tissues undergoing large deformations and perfusion with potentially moderate flow rates~\cite{Badia20097986,Bukac2015197,Armstrong2016,Zakerzadeh2019101}. Ultimately, this has called for more general formulations obeying the fundamental principles of continuum mechanics and thermodynamics~\cite{Coussy2004}, which renders these models applicable to a broader range of scenarios. 

Among various advances, we particularly highlight the development of a general, thermodynamically consistent poromechanics model by Chapelle and Moireau~\cite{Chapelle201482}, which also serves as basis for this work. The model is based on a thermodynamic derivation combined with thermodynamically consistent constitutive laws. It couples the balance of linear momentum for the solid and fluid phases including the viscous dissipation governed by the interaction of both phases due to friction, as well as the conservation of mass. Most importantly, in contrast to the classical quasi-static Biot consolidation model, the aforementioned model satisfies an energy-dissipation identity, predicting the dissipation of the combination of the kinetic and Helmholtz free energy. A further difference between the two approaches is that the former considers the absolute fluid velocity instead of the relative one.

The analysis of the well-posedness, stability and numerical approximation of this class of poromechanics models is still largely open. Among recent advances, we highlight the development and analysis of an implicit time discretization preserving the dissipation-energy identity at the discrete level~\cite{Chapelle201482}; an energy-preserving implicit-explicit time discretization incorporating a (non-iterative) operator splitting, decoupling solid and flow computations~\cite{Burtschell2017effective}; an energy-stable space and time discretization for a linearized model with focus on quasi-incompressible solids~\cite{Burtschell201928}; and finally, a space and time discretization for the same linearized model, exploiting a generalized saddle point structure and ultimately suggesting the use of Taylor-Hood type finite elements~\cite{BARNAFI2020}. 

Motivated by the success of block-partitioned solvers for the related, classical quasi-static Biot equations, the main objective of this work is to develop and analyze for the first time iterative coupling strategies for the general, thermodynamically consistent poromechanics model proposed in \cite{Chapelle201482}. 
Similar to previous theoretical works in this context, see for example \cite{Burtschell201928,BARNAFI2020}, a linearized model is considered for the numerical analysis.

In general, solvers decoupling different physics allow the employment of methods tailored to the separate sub-problems, as flow and elasticity. However, a sequential-implicit solution requires iterating until convergence at each time step. In contrast, fully-implicit approaches, solving the fully-coupled problem at once, yield  unconditional stability but require advanced and efficient preconditioners. Here it is worth to mention that robust iterative coupling strategies can effectively guide the design of scalable preconditioners for the monolithic solution by Krylov subspace methods. 

For robust iterative coupling, in general, a problem-specific strategy is required; yet, we can learn from the well-studied, related Biot equations. For the latter, solvers based upon a sequential-implicit solution of the flow and mechanics sub-problem have been studied since over two decades~\cite{Settari1998219}. The most popular iterative schemes are the undrained split \cite{zienkiewicz1988unconditionally} and the fixed-stress split \cite{Settari1998219}, both relying on additional stabilization to one of the sub-problems. Due to suitable choices of stabilization, both have been shown to be unconditionally stable \cite{Mikelic2013,Castelletto20151593,Both2017} with theoretical convergence rates depending on stabilization and model parameters, but independent on mesh properties; inf-sup stability of the discretization even allows robust convergence in the fluid-incompressible and quasi-impermeable regime~\cite{Storvik2019}. Moreover, the fixed-stress split has been successfully generalized to several complex extensions of the quasi-static Biot equations. In view of biomedical applications, we emphasize work on large deformations~\cite{Borregales2019}.
For optimal performance of the iterative solvers, the choice of the stabilization is well-known to be vital. This choice does depend on several factors~\cite{Storvik2019} as problem parameters, but also boundary conditions and geometry, which are difficult to quantify. To alleviate this, it has been shown in~\cite{Both2018} that Anderson acceleration~\cite{walker2011anderson} greatly relaxes the requirement of optimal stabilization. Furthermore, utilizing the fact that stabilized split schemes are equivalent to a preconditioned Richardson iteration \cite{Castelletto20151593}, they provide a basis to design efficient block preconditioners for the fully-implicit approach \cite{White201655}, next to alternative efficient preconditioners~\cite{Phoon2002377,Haga20111466,White2011647,lee2017parameter,adler2017robust}. In this context, the need for optimal stabilization is similarly relaxed. Hence, after all, stabilization parameters derived in theoretical analyses offer a practical choice.

In this work, we develop and analyze splitting schemes for a linearization of the general poromechanics model~\cite{Chapelle201482}, previously introduced and analyzed in~\cite{Burtschell201928,BARNAFI2020}. This (linearized) model resembles Biot's equations, but presents fundamental differences, most importantly, new terms in the momentum equations of the fluid and solid phases due to inertia, and a structurally different saddle-point structure, compared to a double saddle point structure of the Biot equations. Still, iterative coupling concepts can be adapted to the new setting. Ultimately, we present schemes similar to the undrained split and the fixed-stress split. In particular, the undrained-like split is developed and analyzed in the framework of generalized gradient flows and alternating minimization following~\cite{Both2019}, whereas the fixed-stress-like split is understood as block-diagonal $L^2$-type stabilization and analyzed by means of a relative stability analysis. In practice, additional application of Anderson acceleration is suggested, motivated by associated works in the literature and the here presented numerical examples.

This work is structured as follows. In Section \ref{section:model}, we present the general model of interest and its linearized version. In Section \ref{section:undrained-split} and Section \ref{section:fixed-stress}, we present respectively the alternating minimization split and the diagonally $L^2$--stabilized split. The convergence of both schemes is analyzed in Section \ref{section:analysis}. In Section \ref{section:numerical-tests}, an extensive numerical study is presented which validates the theoretical results. Finally, we close with concluding remarks in Section~\ref{section:conclusion}.

\section{The thermodynamically consistent poromechanics model}\label{section:model}


The purpose of this work is to develop efficient solution strategies for the linearized and discretized version of the thermodynamically consistent poromechanics model originally developed by Chapelle and Moireau in \cite{Chapelle201482}, further described below. Two main steps are essential to reach this objective. One is the discretization of the equations (in this work we consider finite difference schemes in time and finite elements for the space discretization) and the other is the linearization of the model through a Newton-Raphson method. It is natural to operate  the linearization on the discrete version of the problem, obtaining a discrete tangent problem to which the solution strategies proposed later on will be applied. This can be named the \emph{discretize then linearize} strategy.

We remark that in the definition of the tangent problem the shape derivatives are neglected, namely the physical domain $\Omega_t$ is identified with the reference one $\Omega_0$. As in this case the tangent problem depends only on the Fr\'ech\'et derivatives of the mathematical operators that govern the nonlinear problem, the discrete tangent problem obtained by means of the discretize then linearize approach is equivalent to the one that would be derived from the \emph{linearize then discretize} strategy. The latter strategy corresponds to address the linearization of the continuous problem first, giving rise to a fully continuous tangent problem. Then, we address the numerical discretization of such problem and we develop the numerical solvers for it, based on the splitting into several sub-problems. 
We choose to follow the latter approach, because it is much simpler as it allows us to work with the strong formulation of the equations.


\subsection{The general model for finite deformations}

The model assumes that two phases, a fluid ($f$) and a solid ($s$), coexist at each point of the domain of interest. Let us denote by $\phi$ the volume fraction of the fluid. We use Lagrangian (reference) and Eulerian (physical) coordinate frames, denoting by $\Omega_0$ the domain in the Lagrangian frame and by $\Omega_t$ the same domain in the deformed configuration. In the same way, we denote with the subindex ($0$) the operators defined in the Lagrangian frame. For example, given the displacement field in the Lagrangian frame, namely $\us(x_0,t)$ such that $x = x_0 + \us(x_0,t)$ for any $x\in\Omega_t,\,x_0\in\Omega_0$, the deformation gradient tensor is $\FF = \II + \nabla_0 \us$ and its determinant is $J = \mathrm{det} \FF$. We also introduce the symbol $J_s = J(1-\phi)$. One of the primary variables of the model is the added mass $m=\rho_\mathrm{f}(J
\phi-\phi_0)$ that is the fluid mass added to the system due to pore deformation. To define the constitutive laws of the model, we introduce $\Psi(\FF,J_s)$ which is a suitable free energy of the solid.

In view of the linearization of the problem, we formulate the equations on the following abstract form: \textit{Find the velocity of the solid phase $\vs$, the velocity of the fluid phase $\vf$ and the added mass (per unit volume) $m$, such that}
\begin{equation*}
    \Solid(\vv)=0;\quad \Fluid(\vv)=0;\quad \Mass(\vv)=0;
\end{equation*}
where the operators $\Solid(\cdot),\ \Fluid(\cdot),\ \Mass(\cdot)$ correspond to the momentum conservation in the solid and fluid phases, and the mass balance, respectively. More precisely, referring to the strong formulation of the model presented in \cite{Burtschell201928}, the operators $\Solid(\cdot),\ \Fluid(\cdot),\ \Mass(\cdot)$ correspond to the following sub-problems:

\begin{description}
\item Given $\vf,\,m$ and $\ff$ in $\Omega_0$, find $\vs$ in $\Omega_0$ such that
\begin{equation*}
    \Solid(\vv)=\rhos(1-\phi_0)\frac{\partial\vs}{\partial t} - \nabla_0 \cdot \Piola_s + (1-\phi) J \FF^{-T} \nabla_0 p - J\phi^2\permeability^{-1}(\vf-\vs) - \rhos(1-\phi_0) \ff =0
\end{equation*}
complemented by the following constitutive laws
\begin{equation*}
    \Piola_s = \frac{\partial \Psi(\FF,J_s)}{\partial \FF}, \quad p = \frac{\partial \Psi(\FF,J_s)}{\partial J_s}.
\end{equation*}
\item Given $\vs,\,m$ and $\ff$ in $\Omega_t$, find $\vf$ in $\Omega_t$ such that
\begin{align*}
    \Fluid(\vv) &= \frac{1}{J}\frac{d}{dt}(\rho_\mathrm{f} J \phi \vf) + \nabla\cdot \left(\rho_\mathrm{f} \phi \vf \otimes \rho_\mathrm{f} (\vf-\vs)\right) - \nabla \cdot (\phi \Cauchy_f) - \theta \vf + \phi \nabla p \\
    &\quad + \phi^2\permeability^{-1}(\vf-\vs) - \rho_\mathrm{f} \phi \ff =0.
\end{align*}
\item Given $\vs,\,\vf$ in $\Omega_t$ find $m$ in $\Omega_t$ such that
\begin{equation*}
    \Mass(\vv) = \frac{1}{J}\frac{dm}{dt} + \nabla \cdot (\rho_\mathrm{f} \phi (\vf-\vs)) - \theta = 0\,.
\end{equation*}
\end{description}

Here, $\rhos$ and $\rhof$ constitute (spatially and temporally) constant densities of the solid and fluid phases, respectively, $\permeability$ denotes the fluid mobility (absolute permeability divided by the fluid viscosity). Potentially, $\porosity$, $\permeability$, and $\kappa_\mathrm{s}$ are spatially varying, and the source $\f$ is varying in space and time. 

The problem must be complemented by boundary and initial conditions.
For the boundary constraints many options are possible, as discussed for example in \cite{Burtschell201928}.
For the sake of simplicity, we present here only one of the possible variants. 
Let us split the whole boundary $\partial\Omega_t$ into two distinct non-intersecting parts, $\Gamma^D_t$ and $\Gamma^N_t$,
where we enforce Dirichlet and Neumann type conditions, respectively. 
Let $\vs^D,\,\vf^D,\,\tt$, be assigned velocities and traction for boundary conditions,
and let $\vs^0,\vf^0$ be the assigned initial values, under the assumption that $\Omega_t=\Omega_0$ at $t=0$,
We define the boundary and initial conditions as follows,
\begin{alignat*}{4}
    \vs &= \vs^D  && \text{on} &\ & \Gamma^D_0 &&\times (0,T),
    \\
    \vf &= \vf^D  && \text{on} && \Gamma^D_t &&\times (0,T),
    \\
    (\Piola_s - (1-\phi) pJ\FF^{-T}) \nn_0 &= \tt_0 &\qquad & \text{on} && \Gamma^N_0 &&\times (0,T),
    \\
    \phi(\Cauchyf - p\II) \nn &= \tt && \text{on} && \Gamma^N_t &&\times (0,T),
    \\
    \vs &= \vs^0 && \text{in} && \Omega_0 &&\times \{0\},
    \\
    \vf &= \vf^0 && \text{in} && \Omega_0 &&\times \{0\},
    \\
    m &= 0  && \text{in} && \Omega_0 &&\times \{0\}.
\end{alignat*}

\subsection{Derivation of the tangent problem}

Using the previous abstract form of the problem, we formally derive the tangent problem. 
To this purpose, we denote by $\Frech_u{\mathcal{A}}$
the derivative of a generic operator $\mathcal{A}$ with respect to the field $u$. 
We point out that such derivative should account for the classical Fr\'ech\'et derivative of the operator, 
combined with the shape derivatives due to deformations of the domain. 
The central hypothesis in the definition of the tangent problem is that we neglect the shape derivatives, 
limiting ourselves to account for the Fr\'ech\'et ones.
In other words, we identify the physical domain, $\Omega_t$, with the reference one, $\Omega_0$
(and for simplicity we drop the subindices $0,t$, denoting $\Omega_t$ and $\Omega_0$ both by $\Omega$).
In this setting, the nonlinear problem is approximated, at the point $\ovv$, by the following linear problem, called \emph{the tangent problem}:
given $\ovv$, such that the boundary and initial conditions of the nonlinear problem are satisfied, calculate $\dvv$, solution of the following system of linear equations,
\begin{equation*}
    \begin{bmatrix}
    \Frech_{\vs}\Solid & \Frech_{\vf}\Solid & \Frech_{m}\Solid\\
    \Frech_{\vs}\Fluid & \Frech_{\vf}\Fluid & \Frech_{m}\Fluid\\
    \Frech_{\vs}\Mass & \Frech_{\vf}\Mass & \Frech_{m}\Mass
    \end{bmatrix}
    \begin{bmatrix}
    \delta \vs \\ \delta \vf \\ \delta m
    \end{bmatrix}
    =
    -\begin{bmatrix}
    \Solid(\ovv) \\ \Fluid(\ovv) \\ \Mass(\ovv)
    \end{bmatrix},
\end{equation*}
where $\Frech_{\vs}\Solid$ etc.\ represent Fr\'ech\'et derivatives at the point $\ovv$,
and the system must be solved using boundary and initial conditions of the same type of the nonlinear problem, but with homogeneous (null) data.


In \cite{Burtschell201928} an approximate yet explicit expression of the tangent problem is provided.
More precisely, the nonlinear problem is linearized around the configuration at rest,
namely $\ovv=0$. As a result we have $\overline {m}=0$ and $\overline{\phi}=\phi_0 \neq 0$.
Concerning the fluid phase, Newtonian and incompressible behavior is assumed, which yields $\stress_f(\vf) = 2\mu_\mathrm{f} \eps{\vf}$, being $\eps{\boldsymbol v}=\frac12 (\nabla {\boldsymbol v} + \nabla {\boldsymbol v}^T)$ the symmetric deformation gradient.
As in \cite{Burtschell201928}, we denote by $\vs,\,\vf,\,m$ the increments with respect to such state and use an additive decomposition of the free energy, with a Saint-Venant Kirchhoff component for the mechanics and a quadratic potential for the volumetric deformation of the solid phase $J_s$, which reads
\begin{align*}
    \Psi(\FF, J_s) = \frac \lambda 2(\tr{\EE})^2 + \mu\EE:\EE +  \frac{\kappa_\mathrm{s}}{2}\left(\frac{J_s}{1-\phi_0}-1\right)^2,
\end{align*}
where $\EE = \frac{1}{2} \left(\FF^T + \FF + \FF^T\FF\right)$ denotes the Green-Lagrangian strain tensor, also $\mu, \lambda$ are the Lam\'e constants and $\kappa_\mathrm{s}$ is the bulk modulus. Under small deformations we have that $\EE\approx \eps{\us}$ and $J\approx 1 + \DIV\us$, which give
\begin{align*}
    \Piola_s &= \frac{\partial \Psi}{\partial \FF} \approx \Cauchys(\us) =  \mathbb C \eps{\us} = \lambda \tr{\eps \us} + 2\mu \eps \us,\\
    p &= \frac{\partial \Psi}{\partial J_s} \approx \frac{\kappa_\mathrm{s}}{(1-\phi_0)^2}\left(\frac {m}{\rhof} - \DIV\,\us\right), 
\end{align*}
where $\mathbb C$ is a fourth order constant tensor (symmetric, positive definite), known as Hooke tensor. In the linearized setting it is possible to reformulate the problem in terms of the (more commonly used) variable $p$ instead of the added mass.
As a result, the approximate tangent problem for the configuration at rest reads as follows:
\textit{find $\us,\vf,p$ such that}

\begin{subequations}
\label{linearized-model}
\begin{alignat}{2}
 \rhos(1-\porosity) \partial_{tt} \displacement - \DIV\, \Cauchys(\displacement) + (1-\porosity) \GRAD p - \porosity^2 \permeability^{-1} (\fluidvelocity - \partial_t \displacement) &= \rhos(1-\porosity) \f,
 \\[8pt]
 \rhof\porosity \partial_{t} \fluidvelocity - \DIV\, \left(\porosity \stress_\mathrm{f}(\fluidvelocity) \right) - \theta \vf + \porosity\GRAD p + \porosity^2 \permeability^{-1} (\fluidvelocity - \partial_t \displacement) &= \rhof\porosity \f,
\\[3pt]
 \frac{\rhof (1-\porosity)^2}{\kappa_\mathrm{s}} \partial_t p + \DIV\, \left(\rhof \porosity \fluidvelocity\right) + \DIV\,\left( \rhof \left(1-\porosity\right) \partial_t \displacement \right) &= \theta.
\end{alignat}
\end{subequations} 
For simplicity, in what follows we assume $\theta = 0$. The system~\eqref{linearized-model} is closed with appropriate boundary conditions
naturally following from the ones of the nonlinear problem. For the sake of clarity we report them here
\begin{subequations}
\begin{alignat}{4}
  \displacement &= \us^D	&\quad&\text{on }&&\Gamma^D&&\times(0,T), 
  \label{model:bc:mechanics-dirichlet}\\
    \fluidvelocity &= \bm{v}^D	&&\text{on }&&\Gamma^D&&\times(0,T), \label{model:bc:fluid-dirichlet}\\
  \left(\mathbb{C} \eps{\displacement}- (1-\porosity)p \mathbf{I} \right)\n &= \tt&&\text{on }&&\Gamma^N&&\times(0,T), 
  \label{model:bc:mechanics-neumann}\\
  \porosity \left(\stress_\mathrm{f}(\fluidvelocity) - p \mathbf{I}\right) \cdot \n &= \tt &&\text{on }&&\Gamma^N&&\times(0,T), \label{model:bc:fluid-neumann}\\
  \displacement  & =\bm{u}_s^0 &\quad&\text{in }&&\Omega&&\times\{0\}, \label{model:initial:mechanics-displacement}\\
  \partial_t \displacement &= \bm{v}_s^0 &&\text{in }&&\Omega&&\times\{0\}, \label{model:initial:mechanics-velocity}\\
  \fluidvelocity &= \bm{v}_f^0 && \text{in }&&\Omega &&\times \{ 0 \}, \label{model:initial:flow-velocity}\\
  p &= p_0, 	&&\text{in }&&\Omega&&\times\{0\}. \label{model:initial:flow-pressure}  
\end{alignat}
\end{subequations}

\subsection{Numerical approximation of the tangent problem}

We start form the time discretization, based on a simple backward Euler approach.
We will discuss later on how higher order time discretizations are also viable
and the resulting discrete problem maintains its fundamental traits,
such that the numerical solvers developed in what follows will still be applicable.

We consider a partition of the time interval of interest $(0,T)$, given by $0=t_0 < t_1 < ... < t_n < ... < t_N=T$ with, for simplicity, constant time step size $\Delta t \,= t_n - t_{n-1}$. The temporal derivatives within the model~\eqref{linearized-model} are approximated by finite differences
\begin{equation*}
 \partial_{t} \displacement(t_n)  \approx \frac{\displacement^n - \displacement^{n-1}}{\Delta t},\quad
 \partial_{tt} \displacement(t_n) \approx \frac{\displacement^n - 2\displacement^{n-1} + \displacement^{n-2}}{\Delta t^2},\quad
 \partial_t \fluidvelocity(t_n) \approx \frac{\fluidvelocity^n - \fluidvelocity^{n-1}}{\Delta t}.
\end{equation*}
We assume that besides the initial data the first time step has been already determined. From the second time step the fully dynamic linearized model can then be approximated by the Implicit Euler discretization using the above finite difference approximations: \textit{For $n\geq 2$, given $\displacement^{n-1},\displacement^{n-2}, p^{n-1},\fluidvelocity^{n-1}$, find $\displacement^n, p^n, \fluidvelocity^n$ such that}

\begin{subequations}
\label{semi-discrete-linearized-model}
\begin{align}
 \nonumber
 &\rhos(1-\porosity)\frac{\displacement^n - 2 \displacement^{n-1} + \displacement^{n-2}}{\Delta t^2} - \DIV\, \Cauchys(\displacement^n) + (1-\porosity) \GRAD p^n
 - \porosity^2 \permeability^{-1} \left(\fluidvelocity^n - \frac{\displacement^n - \displacement^{n-1}}{\Delta t} \right) \\
 &\qquad = \rhos (1-\porosity)\f^n,\\[8pt]
 &\rhof\porosity\frac{\fluidvelocity^n - \fluidvelocity^{n-1}}{\Delta t} - \DIV\,\left( \porosity \Cauchyf(\fluidvelocity^n) \right) + \porosity\GRAD p^n 
 + \porosity^2 \permeability^{-1} \left(\fluidvelocity^n - \frac{\displacement^n - \displacement^{n-1}}{\Delta t} \right) =  \rhof\porosity \f^n,\hspace{1cm}& \\[8pt]
 &\frac{(1-\porosity)^2}{\kappa_\mathrm{s}} \frac{p^n - p^{n-1}}{\Delta t} + \DIV\, \left( \porosity \fluidvelocity^n\right) + \DIV\,\left( \left(1-\porosity\right) \frac{\displacement^n - \displacement^{n-1}}{\Delta t}\right) = 0.
\end{align}
\end{subequations}

Such problem must satisfy the same boundary conditions of \eqref{linearized-model} at each time $t_n$,
where $\f^n$, $\bm{u}_\mathrm{s}^{D,n}$ etc. denote suitable approximations of the external problem data at time $t_n$.
In what follows we will apply the lifting technique to nonhomogeneous Dirichlet boundary data
(in other words, a change of variable is introduced, by subtracting from the solution a function that is regular enough and equal to the prescribed datum on the boundary, such that the problem for the new variable is transformed into a standard homogeneous Dirichlet-type problem).
In this way, all the forcing terms of the problem (volume forces and surface forces/data) 
will be implicitly represented in the volume term $\f^n$ without significant loss of generality.
Initial conditions are prescribed as in~\eqref{model:initial:mechanics-displacement}-\eqref{model:initial:flow-pressure} by suitably approximating the initial data.
Finally, we stress that the mass conservation equation has been divided by the constant fluid density $\rhof$ in order to highlight an apparent symmetry between the equations.

\begin{remark}[Higher order time discretization]
 Applying alternative diagonally implicit Runge-Kutta schemes results in coupled systems of governing equation of similar type as~\eqref{semi-discrete-linearized-model}. Material parameters possibly have to be scaled appropriately, and the right hand side source terms may then also include further previous data. 
 However, we stress that the analysis of the splitting in this work does not depend on the choice of the time discretization similarly as in~\cite{Bause2019}
 and it could possibly be used also in the framework of space-time finite elements, used for example in \cite{BAUSE2017745} for the approximation of Biot poroelasticity system.
\end{remark}


Let $\spaceU$, $\spaceV$, $\spaceP$ denote suitable function spaces for the solid displacement, fluid velocity, and fluid pressure, respectively, at discrete time $t^n$, incorporating in particular homogeneous essential boundary conditions on the relevant boundaries,
\begin{align*}
 \spaceU &:= \left\{ \testu \in \vec H^1(\Omega)^d \, \big| \, (1-\porosity)\testu \in \vec H(\mathrm{div};\Omega),\ \testu = \bm{0}\text{ on }\MechanicsDirichletBoundary \right\},\\
 \spaceV &:= \left\{ \testv \in \vec H^1(\Omega)^d \, \big| \, \porosity\testv \in \vec H(\mathrm{div};\Omega),\ \testv = \bm{0} \text{ on }\FlowDirichletBoundary \right\},\\
 \spaceP &:= L^2(\Omega).
\end{align*}
\begin{remark}
We note that in the weak formulation the constraints $(1-\phi)\testu, \phi\testv \in \vec H(\mathrm{div}; \Omega)$ are formally required for the corresponding terms to be well-defined. This is in fact a regularity condition on $\phi$, where it is sufficient to consider that both $\phi$ and $1/\phi$ belong to $W^{s,r}(\Omega)$ with $s>d/r$ and $r\geq 1$. More details in~\cite{BARNAFI2020}. 
\end{remark}
Then the canonical weak formulation of~\eqref{linearized-model} reads: \textit{Find $(\displacement^n,\fluidvelocity^n,p^n)\in \spaceU  \times \spaceV \times \spaceP$ such that for all test functions $(\testu,\testv,\testp)\in \testspaceU \times \testspaceV \times \testspaceP$ it holds that}
\begin{subequations}
\label{semi-discrete-three-field}
\begin{align}
 \llangle \rhos(1-\porosity)\frac{\displacement^n -2\displacement^{n-1} + \displacement^{n-2}}{\Delta t^2}, \testu \rrangle  + \llangle \mathbb{C} \, \eps{\displacement^n}, \eps{\testu} \rrangle - \llangle p^n, \DIV \left((1-\porosity) \testu \right) \rrangle&\\
 - \llangle \porosity^2 \permeability^{-1} \left(\fluidvelocity^n - \frac{\displacement^n - \displacement^{n-1}}{\Delta t} \right), \testu \rrangle &= \llangle \f_\mathrm{s}^n, \testu \rrangle ,\nonumber\\[8pt]
 \llangle \rhof\porosity\frac{\fluidvelocity^n - \fluidvelocity^{n-1}}{\Delta t}, \testv \rrangle +  \llangle \porosity 2\mu_\mathrm{f}\, \eps{\fluidvelocity^n}, \eps{\testv} \rrangle  - \llangle p^n, \DIV \left(\porosity\testv \right) \rrangle &\\
 + \llangle \porosity^2 \permeability^{-1} \left(\fluidvelocity^n - \frac{\displacement^n - \displacement^{n-1}}{\Delta t} \right), \testv \rrangle  &= \llangle {\f}_\mathrm{f}^n, \testv \rrangle,\nonumber\\[3pt]
 \llangle \frac{(1-\porosity)^2}{\kappa_\mathrm{s}} \frac{p^n - p^{n-1}}{\Delta t}, \testp \rrangle  + \llangle \DIV\, \left( \porosity \fluidvelocity^n +  \left(1-\porosity\right) \frac{\displacement^n - \displacement^{n-1}}{\Delta t} \right), \testp \rrangle  &= 0.
\end{align}
\end{subequations}

The numerical discretization in space is based on the Galerkin projection of the solution $(\displacement^n,\fluidvelocity^n,p^n)\in \spaceU  \times \spaceV \times \spaceP$ on suitable discrete finite element spaces $\spaceUh,\,\spaceVh,\,\spacePh$ that for the sake of simplicity are assumed to be conforming, namely $\spaceUh\subset\spaceU,\,\spaceVh\subset\spaceV,\,\spacePh\subset\spaceP$. 
Also, all the physical parameters of the tangent problem are assumed to be constant in time and uniformly bounded in space. 
Under these assumptions, the fully discrete version of the problem is formally equivalent to \eqref{semi-discrete-three-field}, where the solution $(\bm{u}_{\mathrm{s},h}^n,\bm{v}_{\mathrm{f},h}^n,p_h^n)$ is sought in $\spaceUh  \times \spaceVh \times \spacePh$ 
and the test functions are taken in the same discrete spaces. Then, to avoid redundancy of notation, we will identify problem \eqref{semi-discrete-three-field} with the fully discrete one and we will omit to specify the subindex $h$, unless strictly necessary.
The finite element spaces will be kept generic throughout the derivation of the numerical solution algorithms, until the discussion of suitable numerical examples that will refer to precise choices of such spaces.

\section{A two-way split inspired by alternating minimization\label{section:undrained-split}}


In the following, we introduce an iterative splitting for the semi-discrete approximation~\eqref{semi-discrete-three-field}, decoupling the momentum equation for the solid phase and the remaining two equations -- the method will be directly applicable for the fully discrete approximation. The systematic construction (and later analysis) of the decoupling scheme is based on the general framework introduced in~\cite{Both2019}. The central idea is to first equivalently rewrite the semi-discrete approximation as an auxiliary convex minimization problem, and second apply alternating minimization to derive a robust block-partitioned solver. Ultimately, reformulated in terms of~\eqref{semi-discrete-three-field}, the final scheme is closely related to the \emph{undrained split} for the quasi-static Biot equations~\cite{Kim2011}, adding a \textit{div-div} stabilization term to the momentum equation for the solid phase.

In what follows we require the following assumption, which has two modeling consequences: On one side, it rules out the possibility of considering the incompressible limit ($\kappa_\mathrm{s}\to\infty$) with this approach, and on the other one it imposes that the domain cannot be composed only of fluid ($\phi\neq 1$).
\begin{assumption}\label{assumption-compressibility}
 It holds $\frac{1}{N}:=\frac{(1-\porosity)^2}{\kappa_\mathrm{s}}>0$ almost everywhere in $\Omega$. 
\end{assumption}

\subsection{Problem formulation as convex minimization}

We choose $\displacement^n$ and $\fluidvelocity^n$ as primary variables. Under Assumption~\ref{assumption-compressibility}, the mass conservation equation can be inverted with respect to the pressure, such that
\begin{align}
\label{pressure-definition-semi-discrete}
 p^n = N \left(\Delta t \,g_p^n - \Delta t \,\DIV\, \left( \porosity \fluidvelocity^n\right) - \DIV\,\left( \left(1-\porosity\right) \displacement^n \right) \right),
\end{align}
where
\begin{align*}
 g_p^n:= \frac{ (1-\porosity)^2}{\kappa_\mathrm{s}} \frac{1}{\Delta t} p^{n-1} + \frac{1}{\Delta t}\DIV\,\left( \left(1-\porosity\right) \displacement^{n-1} \right).
\end{align*}
This allows to formally reduce \eqref{semi-discrete-three-field} to a two-field formulation for the solid displacement and fluid velocity: \textit{Find $(\displacement^n,\fluidvelocity^n)\in \spaceU \times \spaceV$ such that for all test functions $(\testu,\testv)\in \testspaceU \times \testspaceV$ it holds that}
\begin{subequations}
\label{semi-discrete-two-field}
\begin{align}
 &\llangle \frac{\rhos(1-\porosity)}{\Delta t^2} \displacement^n, \testu \rrangle  + \llangle \mathbb{C}\,\eps{\displacement^n}, \eps{\testu} \rrangle - \llangle \porosity^2 \permeability^{-1} \left(\fluidvelocity^n - \frac{1}{\Delta t} \displacement^n\right), \testu \rrangle\\
 \nonumber
 &\qquad+ N \llangle - \Delta t\, g_p^n + \Delta t \,\DIV\, \left( \porosity \fluidvelocity^n\right) + \DIV\,\left( \left(1-\porosity\right) \displacement^n \right), \DIV \left((1-\porosity) \testu \right) \rrangle  = \llangle {\g}_\mathrm{s}^n, \testu \rrangle ,\\[8pt]
 & \llangle \porosity \fluidvelocity^n, \testv \rrangle +  \Delta t\llangle \porosity 2\mu_\mathrm{f} \, \eps{\fluidvelocity^n}, \eps{\testv} \rrangle + \Delta t\llangle \porosity^2 \permeability^{-1} \left(\fluidvelocity^n - \frac{1}{\Delta t} \displacement^n \right), \testv \rrangle \\
 \nonumber
 &\qquad + N \llangle -\Delta t\, g_p^n + \Delta t \,\DIV\, \left( \porosity \fluidvelocity^n\right) + \DIV\,\left( \left(1-\porosity\right) \displacement^n \right), \Delta t \DIV \left(\porosity\testv \right) \rrangle  = \Delta t\llangle {\g}_\mathrm{f}^n, \testv \rrangle,
 \nonumber
\end{align}
\end{subequations}
where the momentum equation for the fluid has been scaled by $\Delta t$, and $\g_\mathrm{s}^n\in \testspaceU^\star$ and $\g_\mathrm{f}^n\in \testspaceV^\star$ are defined by
\begin{alignat*}{2}
 \llangle {\g}_\mathrm{s}^n, \testu \rrangle &:= \llangle \f_\mathrm{s}^n, \testu \rrangle + \llangle \frac{\rhos(1 - \porosity)}{\Delta t^2} \left( 2 \displacement^{n-1} - \displacement^{n-2} \right), \testu \rrangle + \llangle \frac{\porosity^2 \permeability^{-1}}{\Delta t} \displacement^{n-1}, \testu \rrangle,&\quad&\testu\in\testspaceU,\\
 \llangle {\g}_\mathrm{f}^n, \testv \rrangle &:= \llangle \f_\mathrm{f}^n, \testv \rrangle + \llangle \porosity \fluidvelocity^{n-1} , \testv \rrangle - \llangle \porosity^2 \permeability^{-1} \displacement^{n-1}, \testv \rrangle ,&&\testv\in\testspaceV.
\end{alignat*}

The symmetry and uniform coercivity of the governing equations~\eqref{semi-discrete-two-field} identify those as the optimality conditions of a block-separable convex minimization problem. Namely it holds
\begin{align}
\label{minimization-formulation-semi-discrete-linearized-model}
 (\displacement^n,\fluidvelocity^n) = \argmin_{(\displacement,\fluidvelocity)\in \spaceU \times \spaceV} \, \discreteenergy(\displacement,\fluidvelocity),
\end{align}
with the energy given by
\begin{align}
\nonumber
 \discreteenergy(\displacement,\fluidvelocity) 
 &:=
 \frac{1}{2} \llangle \frac{\rhos(1-\porosity)}{\Delta t^2} \displacement, \displacement \rrangle
 +
 \frac{1}{2} \llangle \mathbb{C}\,\eps{\displacement}, \eps{\displacement} \rrangle\\
 \nonumber
 &\quad 
 +
 \frac{1}{2}\llangle \rhof\porosity \fluidvelocity, \fluidvelocity \rrangle +  \frac{\Delta t}{2} \llangle \porosity 2\mu_\mathrm{f} \, \eps{\fluidvelocity}, \eps{\fluidvelocity} \rrangle \\
 \nonumber
 &\quad+
 \frac{N}{2} \left\| \Delta t \,g_p^n - \Delta t \,\DIV\, \left( \porosity \fluidvelocity\right) - \DIV\,\left( \left(1-\porosity\right) \displacement \right) \right\|^2 \\
 \nonumber
 &\quad +
 \frac{\Delta t}{2} \llangle \porosity^2 \permeability^{-1} \left(\fluidvelocity - \frac{1}{\Delta t} \displacement \right), \left(\fluidvelocity - \frac{1}{\Delta t} \displacement \right) \rrangle \\
 \label{definition:objective-function}
 &\quad-
 \llangle {\g}_\mathrm{s}^n, \displacement \rrangle
 -
 \Delta t \llangle {\g}_\mathrm{f}^n, \fluidvelocity \rrangle.
\end{align}

The formulation
\eqref{minimization-formulation-semi-discrete-linearized-model}-\eqref{definition:objective-function}
serves as basis for the succeeding construction of a robust split scheme for~\eqref{semi-discrete-three-field}.

\subsection{Robust splitting via alternating minimization\label{section:undrained-split-alternating-minimization}}

Following the approach of~\cite{Both2019}, we propose an iterative block-partitioned solver for the problem \eqref{semi-discrete-three-field}. In particular, the fundamental alternating minimization algorithm is applied to the equivalent variational 
formulation~\eqref{minimization-formulation-semi-discrete-linearized-model}, cf.\  Alg.~\ref{algorithm:undrained-split} for the definition of a single iteration with index $k$. By construction, the approximate solution consecutively minimizes the system energy $\mathcal{J}$.
\begin{algorithm}[!ht]
 \caption{Iteration $k\geq 1$ of the alternating minimization applied to~\eqref{minimization-formulation-semi-discrete-linearized-model}}
 \label{algorithm:undrained-split}
 \normalsize
 \SetAlgoLined
 \DontPrintSemicolon
 
  \vspace{0.25em}
 
  Input: $(\displacement^{n,k-1},\fluidvelocity^{n,k-1})\in\spaceU \times \spaceV$ \; \vspace{0.25em}
   
  Determine $\displacement^{n,k} := \argmin_{\displacement\in\spaceU}\, \discreteenergy(\displacement,\fluidvelocity^{n,k-1})$\; \vspace{0.25em}

  Determine $\fluidvelocity^{n,k} := \argmin_{\fluidvelocity\in\spaceV}\, \discreteenergy(\displacement^{n,k},\fluidvelocity)$\; \vspace{0.2em}
\end{algorithm}

By introducing a pressure iterate, $p^{n,k}$, analogously to~\eqref{pressure-definition-semi-discrete}
\begin{align*}
  p^{n,k} := N \left(\Delta t \,g_p^n - \Delta t \,\DIV\, \left( \porosity \fluidvelocity^{n,k}\right) - \DIV\,\left( \left(1-\porosity\right) \displacement^{n,k} \right) \right),\quad k\geq 0,
\end{align*}
Alg.~\ref{algorithm:undrained-split} can be equivalently reformulated in the context of the three-field formulation~\eqref{semi-discrete-three-field}. In particular, the $k$-th iteration of the iterative splitting scheme decouples in two steps. In the first step, a \textit{div-div} stabilized momentum equation for the solid phase is solved: \textit{given $(\fluidvelocity^{n,k-1},p^{n,k-1})\in \spaceV \times \spaceP$, find $\displacement^{n,k} \in \spaceU$ satisfying for all $\testu \in \testspaceU$}
\begin{align}
 &\llangle \rhos(1-\porosity)\frac{\displacement^{n,k} -2\displacement^{n-1} + \displacement^{n-2}}{\Delta t^2}, \testu \rrangle  + \llangle \mathbb{C} \, \eps{\displacement^{n,k}}, \eps{\testu} \rrangle\\
 \nonumber
 &\qquad
 + N \llangle \DIV \left((1 - \porosity) (\displacement^{n,k} - \displacement^{n,k-1})\right), \DIV \left((1 - \porosity) \testu \right) \rrangle \nonumber\\
 \label{undrained-split-1}
 &\qquad  - \llangle p^{n,k-1}, \DIV \left((1-\porosity) \testu \right) \rrangle - \llangle \porosity^2 \permeability^{-1} \left(\fluidvelocity^{n,k-1} - \frac{\displacement^{n,k} - \displacement^{n-1}}{\Delta t} \right), \testu \rrangle = \llangle \f_\mathrm{s}^n, \testu \rrangle.
\end{align}
In the second step, the mass conservation and fluid momentum equations are solved: \textit{given $\displacement^{n,k} \in \spaceU$, find $(\fluidvelocity^{n,k},p^{n,k})\in  \spaceV \times \spaceP$, satisfying for all $(\testv,\testp)\in\testspaceV\times\testspaceP$}
\begin{subequations}
 \label{undrained-split-2}
 \begin{align}
%
 &\llangle \rhof\porosity\frac{\fluidvelocity^{n,k} - \fluidvelocity^{n-1}}{\Delta t}, \testv \rrangle +  \llangle \porosity 2\mu_\mathrm{f}\, \eps{\fluidvelocity^{n,k}}, \eps{\testv} \rrangle  - \llangle p^{n,k}, \DIV \left(\porosity\testv \right) \rrangle \\
 &\qquad+ \llangle \porosity^2 \permeability^{-1} \left(\fluidvelocity^{n,k} - \frac{\displacement^{n,k} - \displacement^{n-1}}{\Delta t} \right), \testv \rrangle  = \llangle {\f}_\mathrm{f}^n, \testv \rrangle,\nonumber\\[8pt]
 &\llangle \frac{(1-\porosity)^2}{\kappa_\mathrm{s}} \frac{p^{n,k} - p^{n-1}}{\Delta t}, \testp \rrangle  + \llangle \DIV\, \left( \porosity \fluidvelocity^{n,k} +  \left(1-\porosity\right) \frac{\displacement^{n,k} - \displacement^{n-1}}{\Delta t} \right), \testp \rrangle  = 0,
\end{align}
\end{subequations}
In the remainder of this paper, we refer to the scheme~\eqref{undrained-split-1}--\eqref{undrained-split-2} as the \textit{alternating minimization split}. 

In equation \eqref{undrained-split-1} the term 
$$N \llangle \DIV \left((1 - \porosity) (\displacement^{n,k} - \displacement^{n,k-1})\right), \DIV \left((1 - \porosity) \testu \right) \rrangle$$ 
naturally emerges to stabilize, namely control, the increment of $\displacement^{n}$ at each iteration, such that convergence is guaranteed. This will be proved in more detail in Section~\ref{section:convergence-analysis:undrained-split}.

\section{Diagonally \texorpdfstring{$L^2$}{L2}--stabilized two-way split \label{section:fixed-stress}}

Another prominent class of block-partitioned solvers for coupled problems with saddle-point structure are $L^2$--stabilized splits, which have been successful especially in the context of coupled flow and mechanics. The so-called \textit{fixed-stress split} for the quasi-static Biot equations~\cite{Kim2009,Kim2011}, for instance, decouples solid and flow computations and employs simple $L^2$--stabilization of the mass conservation equation, resulting in unconditional convergence~\cite{Mikelic2013,Both2017}. It is worth mentioning  that in practice the fixed-stress split often is superior to the undrained split~\cite{Kim2009}, also motivating further investigation in the context of this work. Moreover, in the context of thermoporoelasticity diagonal $L^2$--stabilization has been recently investigated for coupled systems consisting of more than two equations~\cite{Brun2020}. In particular, it has been observed that adding stabilization to multiple equations can be beneficial. 

In the following, we present a \textit{diagonally $L^2$--stabilized two-way split} for~\eqref{semi-discrete-three-field}. At first, we allow for stabilization of any of the three equations, introducing three stabilization parameters: $\solidstabilization$ (tensor-valued), $\fluidstabilization$ (tensor-valued), $\pressurestabilization$ (scalar-valued), potentially varying in space.

A single iteration of the splitting scheme is composed of two steps. Let $k\geq1$ denote the iteration index. Following the idea of the fixed-stress approach, the stabilized fluid flow problem is solved first; this is not required for convergence. The $L^2$-stabilized fluid flow step reads:  \textit{given $\displacement^{n,k-1} \in \spaceU$, find $(\fluidvelocity^{n,k},p^{n,k})\in \spaceV \times \spaceP$, satisfying for all $(\testv,\testp)\in\testspaceV\times\testspaceP$}
\begin{subequations}
 \label{fp-s-1}
 \begin{align}
 &\llangle \rhof\porosity\frac{\fluidvelocity^{n,k} - \fluidvelocity^{n-1}}{\Delta t}, \testv \rrangle +  \llangle \porosity 2\mu_\mathrm{f}\, \eps{\fluidvelocity^{n,k}}, \eps{\testv} \rrangle + \llangle \fluidstabilization (\fluidvelocity^{n,k} - \fluidvelocity^{n,k-1}), \testv \rrangle  \\
 \nonumber
 &\qquad - \llangle p^{n,k}, \DIV \left(\porosity\testv \right) \rrangle + \llangle \porosity^2 \permeability^{-1} \left(\fluidvelocity^{n,k} - \frac{\displacement^{n,k-1} - \displacement^{n-1}}{\Delta t} \right), \testv \rrangle  = \llangle {\f}_\mathrm{f}^n, \testv \rrangle,\nonumber\\[8pt]
 &\llangle \frac{(1-\porosity)^2}{\kappa_\mathrm{s}} \frac{p^{n,k} - p^{n-1}}{\Delta t}, \testp \rrangle  + \llangle \pressurestabilization (p^{n,k} - p^{n,k-1}), \testp \rrangle  \\
 \nonumber
 &\qquad +
 \llangle \DIV\, \left( \porosity \fluidvelocity^{n,k}\right), \testp \rrangle + \llangle \DIV\,\left( \left(1-\porosity\right) \frac{\displacement^{n,k-1} - \displacement^{n-1}}{\Delta t} \right), \testp \rrangle  = 0,
\end{align}
\end{subequations}
The second ($L^2$-stabilized solid mechanics) step reads: \textit{given $(\fluidvelocity^{n,k},p^{n,k})\in \spaceV \times \spaceP$, find $\displacement^{n,k} \in \spaceU$ satisfying for all $\testu \in \testspaceU$}
\begin{align}
 \label{fp-s-2}
 &\llangle \rhos(1-\porosity)\frac{\displacement^{n,k} -2\displacement^{n-1} + \displacement^{n-2}}{\Delta t^2}, \testu \rrangle  + \llangle \mathbb{C} \, \eps{\displacement^{n,k}}, \eps{\testu} \rrangle
 + \llangle \solidstabilization (\displacement^{n,k} - \displacement^{n,k-1}), \testu \rrangle\\
 \nonumber
 &\qquad  - \llangle p^{n,k}, \DIV \left((1-\porosity) \testu \right) \rrangle - \llangle \porosity^2 \permeability^{-1} \left(\fluidvelocity^{n,k} - \frac{\displacement^{n,k} - \displacement^{n-1}}{\Delta t} \right), \testu \rrangle = \llangle \f_\mathrm{s}^n, \testu \rrangle.
\end{align}

Finding physically motivated stabilization parameters $\solidstabilization$, $\fluidstabilization$, $\pressurestabilization$, as for the original fixed-stress split, mainly due to an increased complexity of the model (additional terms associated to dynamics) compared to the quasi-static Biot equations. For the same consideration, a simple discussion utilizing convex minimization and alternating minimization as in Section~\ref{section:undrained-split} is problematic, while being possible for the original fixed-stress split~\cite{Both2019}. However, to get a better intuition, a closer look at \mbox{(skew-)}symmetries of the governing equations~\eqref{semi-discrete-three-field} and (partial) Schur complements is useful; in particular, the skew-symmetry of the $\displacement-p$ coupling motivates positive stabilization of the mass conservation equation, and the symmetry of the $\displacement-\fluidvelocity$ coupling suggests negative stabilization of the momentum equation of the solid phase. Since after all the $\displacement-(\fluidvelocity,p)$ coupling is neither symmetric nor skew-symmetric, these observations merely lead to inaccurate insight. Instead, a succeeding convergence analysis in Section~\ref{fp-s-split} is going to suggest practical, potentially vanishing values for the parameters, which eventually lead to unconditional stability.

\section{\textit{A priori} convergence analysis of the proposed two-way splits}\label{section:analysis}

In this section, we address the \textit{a priori} convergence analysis of both the alternating minimization split~\eqref{undrained-split-1}--\eqref{undrained-split-2} and the diagonally $L^2$--stabilized two-way split~\eqref{fp-s-1}--\eqref{fp-s-2}, proposed in Section~\ref{section:undrained-split} and Section~\ref{section:fixed-stress}, respectively. The two primary goals are to (i) prove the linear convergence of the alternating minimization split, and (ii) determine ranges and specific practical values for the stabilization parameters employed within the diagonally $L^2$--stabilized two-way split ensuring convergence. The two goals will be achieved using different techniques. For item (i) the interpretation of the alternating minimization split as alternating minimization applied to a strongly convex minimization problem is extensively exploited, allowing for the systematic application of sharp abstract convergence results from the literature; for item (ii) a slightly more technical approach is chosen due to the fact that the two-way split~\eqref{fp-s-1}--\eqref{fp-s-2} does not fully conform with any (skew-)symmetry. In particular, we relax the classical (quotient) convergence criterion by means of the root convergence criterion, briefly called r-convergence, see for example \cite{NoceWrig06}. More precisely, we formulate a general convergence criterion (based on relative stability) that turns out to be a sufficient condition for the r-convergence of the proposed iterative method.

\subsection{Convergence analysis of the alternating minimization split for the tangent model\label{section:convergence-analysis:undrained-split}}
 
Guaranteed linear convergence of the alternating minimization split~\eqref{undrained-split-1}--\eqref{undrained-split-2} is a direct consequence of its interpretation as alternating minimization applied to a (strongly) convex optimization problem, cf., Section~\ref{section:undrained-split-alternating-minimization} and e.g.,~\cite{Beck2015}.
Furthermore, using simple yet largely sharp abstract convergence results for alternating minimization in a Banach space setting, cf.~\cite{Both2019b}, an upper bound of the rate of convergence can be provided. In the aforementioned work, it is showed that in each of the two steps of the alternating minimization, the energy values of the iterates are sequentially decreased with the decrease merely governed by convexity and continuity properties of the restricted minimization problems. Since the energy $\discreteenergy$ is quadratic, energy differences relative to the optimum will directly translate to distances to the solution, measured in the problem-specific norm induced by the Hessian of the energy (at an arbitrary point). We define $|\cdot|$ on $\testspaceU \times \testspaceV$ for $(\testu,\testv)\in \testspaceU \times \testspaceV$ by
 \begin{align*}
  \left| (\testu,\testv) \right|^2 &:= \llangle \frac{\rhos(1-\porosity)}{\Delta t^2} \testu, \testu \rrangle + \llangle \mathbb{C} \eps{\testu}, \eps{\testu} \rrangle
+\llangle \rhof \porosity \testv, \testv \rrangle 
+
\Delta t \llangle \porosity 2\mu_\mathrm{f} \eps{\testv}, \eps{\testv} \rrangle\\ 
&\qquad +
N \left\| \Delta t \DIV \left(\porosity \testv \right) + \DIV \left((1-\porosity) \testu \right)\right\|^2\\
&\qquad
+
\Delta t \llangle \porosity^2 \permeability^{-1} \left(\testv - \frac{1}{\Delta t} \testu \right), \left(\testv - \frac{1}{\Delta t} \testu \right) \rrangle.
\end{align*} 

In order to estimate the rate of convergence, we introduce a technical, \textit{a priori} material constant $\gamma\geq 0$, given by 
\begin{align}
 \label{definition:material-constant-undrained-split}
  \gamma := \mathrm{min}\left\{\mathrm{max} \left\{ \gamma_1(\zeta,\eta,\vartheta), \gamma_2(\zeta,\eta,\vartheta) \right\} \, \Big| \, \zeta>0,\  \eta\in[0,1],\ \vartheta \in [0,1] \right\},
\end{align}
where
\begin{subequations}
\label{eq:gamma-1-2}
 \begin{align}
 \gamma_1(\zeta,\eta,\vartheta) :=&
 (1+\zeta^{-1})\eta N \Delta t^2 \left\| \frac{\left|\GRAD\porosity\right|^2}{\rhos(1-\porosity)} \right\|_{L^\infty(\Omega)}
 +
 \vartheta \Delta t \left\| \frac{\porosity^2\kappa_\mathrm{m}^{-1}}{\rhos(1-\porosity)} \right\|_{L^\infty(\Omega)},\\
 \gamma_2(\zeta,\eta,\vartheta) :=&
 (1+\zeta) \frac{N}{K_\mathrm{dr,\porosity,min}} +
 (1+\zeta^{-1})(1-\eta) N C_\mathrm{Korn,1} 
 +
 (1-\vartheta) \frac{C_\mathrm{Korn,2}}{\Delta t},
 \end{align}
 \end{subequations}
 with $\kappa_\mathrm{m}>0$ denoting the smallest eigenvalue of the permeability tensor $\permeability$,
 $K_\mathrm{dr,\porosity,min}>0$ being a porosity dependent bulk modulus type constant given by
 \begin{align}
 \label{eq:bulk-modulus}
  K_\mathrm{dr,\porosity,min}^{-1}
  &:= \left\| (1-\porosity)^2 \mathbf{I} : \mathbb{C}^{-1} : \mathbf{I} \right\|_{L^\infty(\Omega)},
 \end{align}
 and $C_\mathrm{Korn,1},C_\mathrm{Korn,2}>0$ taking on the role of generalized Korn/Poincar\'e constants, 
 defined as the minimum positive numbers such that
 \begin{alignat}{2}
 \label{eq:korn1}
   \llangle \GRAD \porosity^\top \GRAD \porosity \testu, \testu \rrangle
  &\leq 
  C_\mathrm{Korn,1} \llangle \mathbb{C} \eps{\testu}, \eps{\testu} \rrangle,&&\text{ for all }\testu\in\testspaceU,\\
  \llangle \porosity^2 \permeability^{-1} \testu, \testu \rrangle
  &\leq 
 \label{eq:korn2}
 C_\mathrm{Korn,2} \llangle \mathbb{C} \eps{\testu}, \eps{\testu} \rrangle,&&\text{ for all }\testu\in\testspaceU.
 \end{alignat}
 It is fair to assume that $C_\mathrm{Korn,1}$ and $C_\mathrm{Korn,2}$ are closely related to the inverse of the drained bulk modulus $K_\mathrm{dr}:= K_\mathrm{dr,0,min}$.

Finally, focusing only on the fully transient model, the linear convergence result for the alternating minimization split scheme reads as follows.

\begin{theorem}[Linear convergence of the alternating minimization split \label{theorem:convergence-undrained}]
 Let $(\displacement^n,\fluidvelocity^n)\in \spaceU \times \spaceV$, $n\geq 2$, denote the solution to~\eqref{minimization-formulation-semi-discrete-linearized-model}, and let $(\displacement^{n,k},\fluidvelocity^{n,k})\in \spaceU \times \spaceV$, $k \geq 1$, denote the corresponding approximation defined by Alg.~\ref{algorithm:undrained-split}. Let $\gamma\geq 0$ be the material constant defined as in~\eqref{definition:material-constant-undrained-split}. 
 Then, for all $k\geq 1$, it holds that 
 \begin{align*}
  \left|(\displacement^{n,k} - \displacement^n, \fluidvelocity^{n,k} - \fluidvelocity^n)\right|^2 & \leq \left(1 - \frac{1}{1+\gamma} \right)^2  \  \left|(\displacement^{n,k-1} - \displacement^n, \fluidvelocity^{n,k-1} - \fluidvelocity^n)\right|^2.
 \end{align*}
\end{theorem}

The convergence result is similar as for the undrained split for the quasi-static Biot equations, cf.~\cite{Mikelic2013}. In particular, the theoretical result suggests degenerating convergence for nearly incompressible and impermeable media. In contrast to the quasi-static Biot equations, porosity heterogeneities may also affect the performance of the splitting scheme, as the material constants $\gamma_1$ and $\gamma_2$ depend on the spatial gradients of $\porosity$. However, a numerical test in Section~\ref{section:numerics-undrained} does only show a weak influence.

The proof of Theorem~\ref{theorem:convergence-undrained} is a plain application of the following abstract convergence result for the alternating minimization, here specifically formulated in terms of Alg.~\ref{algorithm:undrained-split}. 

\begin{lemma}[Convergence of the alternating minimization~\cite{Both2019b}\label{Lemma:AM_convergence}]
Let $|\cdot|$, $|\cdot|_\mathrm{s}$, and $|\cdot|_\mathrm{f}$ denote semi-norms on $\testspaceU \times \testspaceV$, $\testspaceU$, and $\testspaceV$, respectively, such that:
\begin{itemize}
 \item[$\mathrm{(A_1)}$] There exist $\beta_\mathrm{s},\beta_\mathrm{f}\geq 0$, such that for all $(\testu,\testv)\in\testspaceU \times \testspaceV$ it holds that
\begin{align*}
 |(\testu,\testv)|^2 \geq \beta_\mathrm{s} | \testu 
|_\mathrm{s}^2\qquad \text{and} \qquad 
 |(\testu,\testv)|^2 \geq \beta_\mathrm{f} | \testv |_\mathrm{f}^2
\end{align*}

\end{itemize} 
Let $\discreteenergy:\spaceU \times \spaceV \rightarrow \mathbb{R}$ be Frech\'et differentiable with ${\cal D}\discreteenergy$ denoting its derivative such that:
\begin{itemize}
  \item[$\mathrm{(A_2)}$] The energy $\discreteenergy$ is strongly convex wrt.\ $|\cdot|$ with modulus $\sigma>0$, {i.e., for all $\displacement,\bardisplacement\in\spaceU$ and $\fluidvelocity,\barfluidvelocity\in\spaceV$ it holds that 
  \begin{align*}
  \discreteenergy(\bardisplacement,\barfluidvelocity) \geq \discreteenergy(\displacement,\fluidvelocity) + \llangle {\cal D}\discreteenergy(\displacement,\fluidvelocity), (\bardisplacement - \displacement,\barfluidvelocity - \fluidvelocity)  \rrangle + \frac{\sigma}{2} |(\bardisplacement - \displacement,\barfluidvelocity - \fluidvelocity)|^2.
  \end{align*}}
  \item[$\mathrm{(A_3)}$] The partial functional derivatives ${\cal D}_{\displacement} \discreteenergy$ and 
${\cal D}_{\fluidvelocity} \discreteenergy$ are uniformly Lipschitz continuous wrt.\ 
$|\cdot|_\mathrm{s}$ and $|\cdot|_\mathrm{f}$ with Lipschitz constants $L_\mathrm{s}$ and 
$L_\mathrm{f}$, respectively, {i.e., for all $(\displacement,\fluidvelocity) \in \spaceU 
\times \spaceV$ and $(\testu,\testv) \in \testspaceU \times 
\testspaceV$ it holds that 
  \begin{align*}
   \discreteenergy(\displacement + \testu, \fluidvelocity) &\leq \discreteenergy(\displacement, \fluidvelocity) + \llangle {\cal D}_{\displacement} \discreteenergy(\displacement, \fluidvelocity),\testu \rrangle + \frac{L_\mathrm{s}}{2} \left| \testu \right|_\mathrm{s}^2\,,\\[1ex]
   \discreteenergy(\displacement, \fluidvelocity + \testv) &\leq \discreteenergy(\displacement, \fluidvelocity) + \llangle {\cal D}_{\fluidvelocity} \discreteenergy(\displacement, \fluidvelocity),\testv \rrangle + \frac{L_\mathrm{f}}{2} \left| \testv \right|_\mathrm{f}^2\,.
  \end{align*}}
\end{itemize}
\noindent
Let $(\displacement^n,\fluidvelocity^n)\in \spaceU \times \spaceV$ denote the unique
solution to~\eqref{minimization-formulation-semi-discrete-linearized-model}, and let $(\displacement^{n,k},\fluidvelocity^{n,k}) \in \spaceU \times \spaceV$ 
denote the corresponding approximation defined by Alg.~\ref{algorithm:undrained-split}. 
Then, for all $k\geq 1$ it follows that 
\begin{align*}
 &\discreteenergy(\displacement^{n,k},\fluidvelocity^{n,k}) - 
\discreteenergy(\displacement^n,\fluidvelocity^n)
\leq 
\left( 1 - \frac{\beta_\mathrm{s} \sigma}{L_\mathrm{s}} \right) \left( 1 
- \frac{\beta_\mathrm{f} \sigma}{L_\mathrm{f}} \right) \, \left( 
\discreteenergy(\displacement^{n,k-1},\fluidvelocity^{n,k-1}) - \discreteenergy(\displacement^n,\fluidvelocity^n) 
\right).
\end{align*}
\end{lemma}

With this, we are able to prove Thm.~\ref{theorem:convergence-undrained}.

\begin{proof}[Proof of Thm.~\ref{theorem:convergence-undrained}]
In order to apply Lemma~\ref{Lemma:AM_convergence}, we verify the conditions $\mathrm{(A_1)}$--$\mathrm{(A_3)}$.
First of all, we note that the energy $\discreteenergy$ is quadratic. Since $|\cdot|$ is induced by the Hessian of $\discreteenergy$, i.e., 
\begin{alignat}{2}
 \label{Eq:proof-aux:second-derivative-energy}
  \left| (\testu,\testv) \right|^2 
  :=&
  \llangle D^2 \mathcal{J}(\displacement,\fluidvelocity) (\testu, \testv), (\testu, \testv) \rrangle,&& \ (\testu,\testv) \in \testspaceU \times \testspaceV,
\end{alignat}
(for arbitrary $(\displacement,\fluidvelocity) \in \testspaceU \times \testspaceV$),
the convexity property $\mathrm{(A_2)}$ is satisfied with $\sigma=1$.

Similarly, by defining $|\cdot|_\mathrm{s}$ and $|\cdot|_\mathrm{f}$ on $\testspaceU$ and $\testspaceV$, respectively, as partial Hessians of $\discreteenergy$
\begin{alignat*}{2}
  \left| \testu \right|_\mathrm{s}^2 
  :=&
  \llangle {\cal D}_{\displacement}^2 \mathcal{J}(\displacement,\fluidvelocity) \testu, \testu \rrangle,&& \ \testu \in \testspaceU  \\
   \left| \testv \right|_\mathrm{f}^2
  :=&\llangle {\cal D}_{\fluidvelocity}^2 \mathcal{J}(\displacement,\fluidvelocity) \testv, \testv \rrangle, && \ \testv\in\testspaceV, 
\end{alignat*}
(for arbitrary $(\displacement,\fluidvelocity) \in \testspaceU \times \testspaceV$),
the smoothness property $\mathrm{(A_3)}$ is satisfied with $L_\mathrm{s}=L_\mathrm{f}=1$. 

It remains to examine $\mathrm{(A_1)}$. In the following, we show that one can choose $\beta_\mathrm{s}=\beta_\mathrm{f} = (1+\gamma)^{-1}$, i.e., it holds 
\begin{subequations}
\label{proof:undrained-convergence:aux-1}
\begin{alignat}{2}
\label{proof:undrained-convergence:aux-1a}
  |\testu|_\mathrm{s}^2 \leq (1+\gamma) |(\testu, \testv)|^2,&&\quad \text{for all }(\testu,\testv) \in \testspaceU \times \testspaceV,\\
\label{proof:undrained-convergence:aux-1b}
  |\testv|_\mathrm{f}^2 \leq (1+\gamma) |(\testu, \testv)|^2,&&\quad \text{for all }(\testu,\testv) \in \testspaceU \times \testspaceV.
\end{alignat}
\end{subequations}
For both estimates, the following inequality will be of help
\begin{align}
\label{proof:undrained-convergence:aux-2}
 T^\star&:= N \underbrace{\left\| \DIV \left((1-\porosity) \testu \right) \right\|^2}_{=: T_1}
  +
  \frac{1}{\Delta t} \underbrace{\llangle \porosity^2 \permeability^{-1} \testu, \testu \rrangle}_{=:T_2}\\
  &
  \nonumber
  \leq 
  \gamma \left( \llangle \frac{\rhos(1-\porosity)}{\Delta t^2} \testu, \testu \rrangle
  +
  \llangle \mathbb{C} \eps{\testu}, \eps{\testu} \rrangle  \right).
\end{align}
Indeed, for $T_1$, using the product rule, the Cauchy-Schwarz inequality and Young's inequality, we obtain for all $\zeta>0$ 
 \begin{align}
 \label{proof:ud-aux3a}
  T_1& \leq
  (1+\zeta) \underbrace{\left\| (1-\porosity) \DIV \testu \right\|^2}_{=:T_1'}
    +
  (1+\zeta^{-1}) \underbrace{\left\| \GRAD \porosity \cdot \testu \right\|^2}_{=:T_1''}.
 \end{align}
 Further, employing the definitions of $K_\mathrm{dr,\porosity,min}$ and $C_\mathrm{Korn,1}$, see~\eqref{eq:bulk-modulus} and~\eqref{eq:korn1}, it follows that
\begin{subequations}
\label{proof:ud-aux3b}
 \begin{align}
  T_1' &\leq \frac{1}{K_\mathrm{dr,\porosity,min}} \llangle \mathbb{C} \eps{\testu}, \eps{\testu} \rrangle, \\
  T_1''
  &\leq  
  \Delta t^2 \left\| \frac{\GRAD \porosity^\top \GRAD\porosity}{\rhos(1-\porosity)} \right\|_{L^\infty(\Omega)} \llangle \frac{\rhos (1-\porosity)}{\Delta t^2} \testu, \testu \rrangle,\\
  T_1'' 
  &\leq 
  C_\mathrm{Korn,1} \llangle \mathbb{C} \eps{\testu}, \eps{\testu} \rrangle.
 \end{align}
\end{subequations}
 Similarly, employing the definitions of $\kappa_\mathrm{m}$, the smallest eigenvalue of $\permeability$, and $C_\mathrm{Korn,2}$, see~\eqref{eq:korn2}, for $T_2$ it holds
 \begin{subequations}
 \label{proof:ud-aux3c}
 \begin{align}
  T_2 &\leq \Delta t^2 \left\| \frac{\porosity^2\kappa_\mathrm{m}^{-1}}{\rhos(1-\porosity)} \right\|_{L^\infty(\Omega)} \llangle \frac{\rhos(1-\porosity)}{\Delta t^2} \testu, \testu \rrangle,\\
  T_2 &\leq C_\mathrm{Korn,2} \llangle \mathbb{C} \eps{\testu}, \eps{\testu} \rrangle.
 \end{align}
 \end{subequations}
 By combining~\eqref{proof:ud-aux3a}--\eqref{proof:ud-aux3c}, balancing the different upper bounds for $T_1''$ and $T_2$, and employing the definitions of $\gamma_1$ and $\gamma_2$, cf.~\eqref{eq:gamma-1-2}, we obtain for all $\zeta>0$, $\eta\in[0,1]$ and $\theta\in[0,1]$
 \begin{align*}
  T_1 + T_2 \leq \gamma_1(\zeta,\eta,\theta) \llangle \frac{\rhos (1-\porosity)}{\Delta t^2} \testu, \testu \rrangle 
  +
  \gamma_2(\zeta,\eta,\theta) \llangle \mathbb{C} \eps{\testu}, \eps{\testu} \rrangle,
 \end{align*}
 and thereby~\eqref{proof:undrained-convergence:aux-2} follows.
 
Finally, we show~\eqref{proof:undrained-convergence:aux-1}. By definition of $|\cdot|_\mathrm{s}$ it holds that
\begin{align*}
  \left| \testu \right|_\mathrm{s}^2 
  =&
  \llangle \frac{\rhos(1-\porosity)}{\Delta t^2} \testu, \testu \rrangle
  +
  \llangle \mathbb{C} \eps{\testu}, \eps{\testu} \rrangle 
  +
  T^\star.
 \end{align*}
 Hence,~\eqref{proof:undrained-convergence:aux-1a} follows from~\eqref{proof:undrained-convergence:aux-2}. By (i) definition of $|\cdot|_\mathrm{f}$, (ii) suitable addition with zero, and application of the Cauchy-Schwarz and Young's inequalities, and (iii)~\eqref{proof:undrained-convergence:aux-2}, it holds
 \begin{align*}
  \left| \testv \right|_\mathrm{f}^2
  &\underset{(i)}{=}
  \llangle \rhof \porosity \testv, \testv \rrangle
  +
  \Delta t \llangle \porosity 2 \mu_\mathrm{f} \eps{\testv}, \eps{\testv} \rrangle
  +
  N \left\| \Delta t \DIV \left( \porosity \testv \right) \right\|^2
  +
  \Delta t \llangle \porosity^2 \permeability^{-1} \testv, \testv \rrangle\\
  &\underset{(ii)}{\leq} 
     \llangle \rhof \porosity \testv, \testv \rrangle
  +
  \Delta t \llangle \porosity 2 \mu_\mathrm{f} \eps{\testv}, \eps{\testv} \rrangle
  +
  \left(1 + \gamma \right) \, N \left\| \Delta t \DIV \left(\porosity \testv \right) + \DIV \left((1-\porosity) \testu \right)\right\|^2 \\
  &\qquad + \left(1 + \gamma\right) \,
 \Delta t \llangle \porosity^2 \permeability^{-1} \left(\testv - \frac{1}{\Delta t} \testu \right), \left(\testv -   \frac{1}{\Delta t} \testu \right) \rrangle 
 + 
 \left(1 + \gamma^{-1} \right) T^\star\\
 &\underset{(iii)}{\leq} (1+\gamma) |(\testu, \testv)|^2.
  \end{align*}
 Hence, we obtain~\eqref{proof:undrained-convergence:aux-1b}, and thereby~$\mathrm{(A_1)}$.

 Ultimately, the assumptions of Lemma~\ref{Lemma:AM_convergence} are satisfied, and it follows for all $k\geq 1$ that
 \begin{align*}
   &\discreteenergy(\displacement^{n,k},\fluidvelocity^{n,k}) - 
\discreteenergy(\displacement^n,\fluidvelocity^n)
\leq 
\left( 1 - (1+\gamma)^{-1}\right)^2 \, \left( 
\discreteenergy(\displacement^{n,k-1},\fluidvelocity^{n,k-1}) - \discreteenergy(\displacement^n,\fluidvelocity^n) 
\right).
 \end{align*}
 Moreover, since $\mathcal{J}$ is quadratic, $(\displacement^n,\fluidvelocity^n)$ is a local minimum of 
$\mathcal{J}$, and $|\cdot|$ is induced by the functional Hessian of $\discreteenergy$ 
via~\eqref{Eq:proof-aux:second-derivative-energy}, we have that 
\begin{align*}
 \mathcal{J}(\displacement^{n,k},\fluidvelocity^{n,k}) - \mathcal{J}(\displacement^n,\fluidvelocity^n) = 2 
\left|(\displacement^{n,k} - \displacement^n, \fluidvelocity^{n,k} - \fluidvelocity^n)\right|^2
\end{align*}
for all $k\geq 0$. Thereby, the assertion follows. 
\end{proof}

\subsection{Convergence analysis of the diagonally \texorpdfstring{$L^2$}{L2}--stabilized split for the tangent model
}\label{fp-s-split}


The essence of the diagonally $L^2$--stabilized split~\eqref{fp-s-1}--\eqref{fp-s-2} is the decoupling of the mechanical displacement from the remaining variables (fluid pressure and velocity). Such a split does neither fully conform with a symmetry nor a saddle point structure of the governing equations. In view of a convergence analysis aiming at employing some contraction argument or similar, it therefore cannot be expected that all coupling terms can be simultaneously canceled by suitable testing as often done, cf., e.g.,~\cite{Both2017}. 
To mitigate this complication, the concept of relative stability will be exploited instead, allowing for a simpler discussion of the coupling terms. 
In the following, the analysis is presented in two steps: (i) a central abstract convergence result for positive real-valued sequences satisfying a relative stability property is introduced; (ii) the result is applied to the diagonally $L^2$--stabilized split~\eqref{fp-s-1}--\eqref{fp-s-2} to show \textit{a priori} convergence.

\subsubsection{Abstract convergence criterion based on relative stability}
Consider a real-valued (positive) sequence $\{ x_k\}_k \subset \mathbb{R}_+$ satisfying the stability property:
\begin{align}
\label{convergence-criterion-stability}
 \textit{There exists a constant }c\in(0,\infty)\textit{ such that }c \sum_{i=1}^\infty x_{k+i} \leq x_k\textit{ for all }k\in\mathbb{N},
\end{align}
without any additional requirement for the stability constant $c$. We call this property the \emph{relative stability} criterion for the sequence $\{ x_k\}_k \subset \mathbb{R}_+$.
This criterion ensures \textit{r-linear convergence for subsequences} (still wrt.\ the original sequence), a weaker form of standard r-linear convergence, covering both contractive and certain non-contractive sequences.

\begin{lemma}[r-linear convergence for subsequences\label{convergence-criterion-II}]
 Let $\{x_k\}_k\subset \mathbb{R}_+$ and $c\in(0,\infty)$ satisfy~\eqref{convergence-criterion-stability}. Then there exists a subsequence $\{ x_{k_l} \}_l$ with $1\leq k_{l+1} - k_l\leq \underset{m\in\mathbb{N}}{\mathrm{arg\,min}}\left(cm\right)^{-\frac{1}{m}}$, which converges r-linearly with
 \begin{align*}
  x_{k_l} \leq \left(\underset{m\in\mathbb{N}}{\mathrm{min}}\, \left(\frac{1}{cm}\right)^{\frac{1}{m}} \right)^{k_l} x_0.
 \end{align*}
\end{lemma}
For the proof of Lemma~\ref{convergence-criterion-II}, we state the following auxiliary result.

\begin{lemma}\label{convergence-criterion-I}
 Let $\{x_k\}_k\subset \mathbb{R}_+$ and $c>0$ satisfying~\eqref{convergence-criterion-stability}. Then for any $k\in\mathbb{N}$ and $\varepsilon>0$ there exists some $n \in \{0, 1, ..., \left\lceil \frac{1}{c\varepsilon} \right\rceil\}$ such that $x_{k+n} \leq \varepsilon x_k$.
\end{lemma}

\begin{proof}
 Let $k\in\mathbb{N}$ and $\varepsilon>0$ be arbitrary but fixed. Assume without loss of generality that $x_k>0$. 
 Then the assertion follows by contradiction: Assume it holds
  $x_{k+i} > \varepsilon x_k$, for all $i=1,...,\left\lceil \frac{1}{c\varepsilon} \right\rceil$; we conclude that it holds
 \begin{align*}
  c \sum_{i=1}^\infty x_{k+i} \geq c \sum_{i=1}^{\left\lceil \frac{1}{c\varepsilon} \right\rceil}
  x_{k+i} > c\varepsilon \, \left\lceil \frac{1}{c\varepsilon} \right\rceil \, x_k \geq x_k,
 \end{align*}
 which contradicts~\eqref{convergence-criterion-stability}.
\end{proof}

\begin{proof}[Proof of Lemma~\ref{convergence-criterion-II}]
 The idea of the proof is to employ Lemma~\ref{convergence-criterion-I} and construct a subsequence of $\{x_k\}_k$, which is linearly (first order) quotient converging, and then conclude r-linear convergence wrt.\ the original sequence. Assume without loss of generality that $x_0 > 0$. Let $m\in\mathbb{N}$ such that $cm > 1$, and let $\varepsilon:= \frac{1}{cm} < 1$, such that $\left\lceil \frac{1}{c\varepsilon} \right\rceil = m$. By Lemma~\ref{convergence-criterion-I} there exists some $n_1 \in \{1, ..., m\}$ such that it holds
$
  x_{n_1} \leq \varepsilon x_0.
$  
 Analogously, for any $i=2, ...$, there exists some $n_i \in \{1, ..., m\}$, satisfying
$
  x_{\sum_{j=1}^{i-1} n_j + n_i} \leq \varepsilon x_{\sum_{j=1}^{i-1} n_j}.
$

Next, we define $\{k_l\}_l\subset \mathbb{N}$ by setting $k_l := \sum_{j=1}^l n_j$ for all $l\in\mathbb{N}$. Since $\varepsilon < 1$ and $n_j \leq m$ for all $j$, it holds that
 \begin{align*}
  \varepsilon^\frac{l}{k_l} = \varepsilon^\frac{l}{\sum_{j=1}^l n_j} \leq \varepsilon^\frac{l}{l\, m} = \left(\frac{1}{cm}\right)^\frac{1}{m}.
 \end{align*}
For $\{x_{k_l}\}_l$, we conclude
 \begin{align*}
  x_{k_l} 
  \leq \varepsilon^l x_0 = \left(\varepsilon^\frac{l}{k_l}\right)^{k_l} x_0 
  \leq \left(\left(\frac{1}{cm}\right)^\frac{1}{m}\right)^{k_l} x_0.
 \end{align*}
 for, so far, arbitrary $m\in\mathbb{N}$. Minimizing the right hand side wrt.\ $m$ ultimately yields the assertion.
\end{proof} 

\subsubsection{Convergence analysis of the diagonally \texorpdfstring{$L^2$}{L2}--stabilized split based on the concept of relative stability}
In the following, we establish linear convergence of the diagonally $L^2$--stabilized two-way split~\eqref{fp-s-1}--\eqref{fp-s-2}. The primary aim of the analysis is to determine ranges for the stabilization parameters $\solidstabilization$, $\fluidstabilization$ and $\pressurestabilization$, which \textit{a priori} guarantee convergence; in addition, we are going to suggest a practical (for simplicity of the presentation not necessarily optimally tuned) set of values. The reader interested in the analysis of optimal convergence rate is referred to analogous studies of the fixed-stress split for the quasi-static Biot equations~\cite{Storvik2019}.

For the convergence analysis, the concept of relative stability and r-linear convergence for subsequences introduced in the previous section is applied. Ultimately, the final result states that it is sufficient to stabilize the mass conservation equation along the lines of the fixed-stress split for the quasi-static Biot equations~\cite{Kim2011,Kim2009,Both2017,Storvik2019}, in order to guarantee convergence. Additional destabilization, i.e., negative stabilization, of the momentum equation for the solid phase theoretically improves the convergence speed. Fluid (de-)stabilization does not further improve the convergence rate.  

To ease the presentation of the analysis, we introduce two notations:
\begin{itemize}
    \item[(N1)] \emph{Weighted squares} $\weightedsquare{\cdot}{\bm{A}}$, defined by $\weightedsquare{\bm{\omega}}{\bm{A}}:= \llangle \bm{A} \bm{\omega}, \bm{\omega} \rrangle$,
 where $\bm{\omega}$ can be a tensor-, vector- or scalar-valued function on $\Omega$, and the weight $\bm{A}$ is a (potentially non-positive definite) function on $\Omega$ with adequate dimensionality such that the above definition is well-defined.
 
 \item[(N2)] \emph{Weighted $L^2$ norms} $\|\cdot\|_{\bm{A}}$ for uniformly positive definite $\bm{A}$, defined by $\|\cdot\|_{\bm{A}}^2:= \weightedsquare{\cdot}{\bm{A}}$.

\end{itemize}

Finally, we state the convergence result for the diagonally $L^2$--stabilized split~\eqref{fp-s-1}--\eqref{fp-s-2}.

\begin{theorem}[Relative stability and convergence of the diagonally $L^2$--stabilized two-way split]\label{theorem:l2s-convergence}
Let $\incu{k}:=\displacement^{n,k} - \displacement^{n,k-1}$, $\incv{k}:=\fluidvelocity^{n,k} - \fluidvelocity^{n,k-1}$, and $\incp{k}:= p^{n,k} - p^{n,k-1}$ denote increments for $k\geq 1$. Furthermore, let $K_\mathrm{dr,\porosity,min}$ and $C_\mathrm{Korn,1}$ as defined in~\eqref{eq:bulk-modulus} and~\eqref{eq:korn1}, resp.; let $\delta_1>0$ and $\delta_2\in(0,2)$ be tuning parameters; and let the stabilization parameters satisfy
 \begin{align}
\label{fp-s-stabilization-parameters}
  \solidstabilization \succeq - \frac{\porosity^2 \permeability^{-1}}{2\Delta t}, \quad\qquad 
  \fluidstabilization \succeq \bm{0},\qquad\quad 
  \pressurestabilization \geq \frac{1}{\delta_2\, \Delta t}\left( C_\mathrm{Korn,1}^{1/2} + K_\mathrm{dr,\porosity,min}^{-1/2} \right)^2,
 \end{align}
where $\bm{A} \succeq \bm{B}$ for tensor-valued maps $\bm{A}$ and $\bm{B}$ on $\Omega$ iff.\ $\bm{A} - \bm{B}$ is uniformly positive definite. Then  the scheme~\eqref{fp-s-1}--\eqref{fp-s-2} satisfies a relative stability criterion of type \eqref{convergence-criterion-stability}, namely
 \begin{align}
    \label{fp-s-stability}
    &\sum_{k=m+1}^{\infty} \left[ \frac{1}{\Delta t^2} \| \incu{k} \|_{\rhos(1-\porosity)}^2 
    +
    \left(1 - \frac{\delta_2}{2} \right) \left\| \eps{\incu{k}} \right\|_{\mathbb{C}}^2 \right]\\
    \nonumber
    &\qquad 
    +
    \sum_{k=m+1}^{\infty} \left[
    \| \incv{k}\|_{\rhof\porosity}^2 +  \Delta t\,  \left\| \eps{\incv{k}} \right\|_{\porosity 2\mu_\mathrm{f}}^2 
    +
    \left\| \incp{k} \right\|_{\frac{(1-\porosity)^2}{\kappa_\mathrm{s}}}^2 \right] \\
    &\quad \leq
    \frac{1}{2} \left\| \incu{m} \right\|_{\hatsolidstabilization}^2
    +
    \frac{\delta_1 + \delta_2}{2} \left\| \eps{\incu{m}} \right\|_{\mathbb{C}}^2
    +
    \frac{\Delta t}{2} \left\| \incv{m} \right\|_{\fluidstabilization}^2
    + 
   \frac{\Delta t}{2}  
    \left\| \incp{m} \right\|_{\hatpressurestabilization}^2,\quad \text{for all }m\in\mathbb{N},
    \nonumber
    \end{align}
    where $\hatsolidstabilization$ and  $\hatpressurestabilization$ denote augmented stabilization parameters (introduced for simpler presentation)
 \begin{equation*}
     \hatsolidstabilization := \solidstabilization + \frac{\porosity^2 \permeability^{-1}}{2\Delta t},
     \qquad
     \hatpressurestabilization := \pressurestabilization + \frac{1}{\delta_1\, \Delta t}\left( C_\mathrm{Korn,1}^{1/2} + K_\mathrm{dr,\porosity,min}^{-1/2} \right)^2.
 \end{equation*}
    If $\tfrac{(1-\porosity)^2}{\kappa_\mathrm{s}}$ is uniformly positive,
    subsequences of $\incu{k},\, \incv{k},\, \incp{k}$ r-linearly converge to zero, in the sense of Lemma~\ref{convergence-criterion-II}.
\end{theorem}

\begin{proof} 
 The proof is organized in five steps, starting with governing equations for increments.

 \paragraph{Increment equations} By subtracting~\eqref{fp-s-1}--\eqref{fp-s-2} at iteration $k$ and $k-1$, $k\geq 2$, we obtain
\begin{subequations}
 \label{fp-s-analysis-1}
\begin{align}
\label{fp-s-analysis-1a}
  &\llangle \rhos(1-\porosity)\frac{\incu{k}}{\Delta t^2}, \testu \rrangle  + \llangle \mathbb{C} \, \eps{\incu{k}}, \eps{\testu} \rrangle
  +   \llangle \solidstabilization(\incu{k} - \incu{k-1}), \testu \rrangle \\
  \nonumber
  &\qquad - \llangle \incp{k}, \DIV \left((1-\porosity) \testu \right) \rrangle - \llangle \porosity^2 \permeability^{-1} \left(\incv{k} - \frac{\incu{k}}{\Delta t} \right), \testu \rrangle = 0,\\[8pt]
\label{fp-s-analysis-1b}
 &\llangle \rhof\porosity\frac{\incv{k}}{\Delta t}, \testv \rrangle +  \llangle \porosity 2\mu_\mathrm{f}\, \eps{\incv{k}}, \eps{\testv} \rrangle
 + \llangle  \fluidstabilization (\incv{k} - \incv{k-1}), \testv \rrangle \\
 \nonumber
 &\qquad- \llangle \incp{k}, \DIV \left(\porosity\testv \right) \rrangle 
 + \llangle \porosity^2 \permeability^{-1} \left(\incv{k} - \frac{\incu{k-1}}{\Delta t} \right), \testv \rrangle  = 0,\\[8pt]
\label{fp-s-analysis-1c}
 &\llangle \frac{(1-\porosity)^2}{\kappa_\mathrm{s}} \frac{\incp{k}}{\Delta t}, \testp \rrangle  
 +  \llangle \pressurestabilization(\incp{k} - \incp{k-1}), \testp \rrangle \\
 \nonumber
 &\qquad + \llangle \DIV\, \left( \porosity \incv{k}\right), \testp \rrangle + \llangle \DIV\,\left( \left(1-\porosity\right) \frac{\incu{k-1}}{\Delta t} \right), \testp \rrangle  
= 0.
\end{align}
\end{subequations}

\paragraph{Testing with current increments}
Testing and summing~\eqref{fp-s-analysis-1} with $\testu=\incu{k}$, $\testv=\Delta t\,\incv{k}$, and $\testp=\Delta t\,\incp{k}$, and finally summing over indices $k=m+1,...,M$, for arbitrary $m<M$, yields
\begin{align}
\nonumber&\sum_{k=m+1}^M \left[ \frac{1}{\Delta t^2} \| \incu{k} \|_{\rhos(1-\porosity)}^2 
+
\left\| \eps{\incu{k}} \right\|_{\mathbb{C}}^2
+
\| \incv{k}\|_{\rhof\porosity}^2 +  \Delta t\,  \left\| \eps{\incv{k}} \right\|_{\porosity 2\mu_\mathrm{f}}^2 
+
\left\| \incp{k} \right\|_{\frac{(1-\porosity)^2}{\kappa_\mathrm{s}}}^2 \right] \\
&\qquad + T_1 = T_2 + T_3
\label{fp-s-analysis-2}
\end{align}
(employing notation (N2) and) with 
\begin{align*}
    T_1 &= \sum_{k=m+1}^M  \left[ 
\llangle \solidstabilization(\incu{k} - \incu{k-1}), \incu{k} \rrangle
+
\Delta t\,  \llangle \fluidstabilization(\incv{k} - \incv{k-1}), \incv{k} \rrangle
+
\Delta t\,  \llangle \pressurestabilization(\incp{k} - \incp{k-1}), \incp{k} \rrangle 
\right],\\
 T_2 &= \sum_{k=m+1}^M   \llangle \incp{k}, \DIV \left((1-\porosity) (\incu{k} - \incu{k-1}) \right) \rrangle, \\
 T_3 &= \Delta t \sum_{k=m+1}^M \left[\llangle \porosity^2 \permeability^{-1} \left(\incv{k}
- \frac{\incu{k}}{\Delta t} \right), \frac{\incu{k}}{\Delta t} \rrangle
 - \llangle \porosity^2 \permeability^{-1} \left(\incv{k} - \frac{\incu{k-1}}{\Delta t} \right), \incv{k} \rrangle \right].
 \end{align*}
We discuss the terms $T_\mathrm{1}$, $T_2$ and $T_3$ separately. For the stabilization term $T_\mathrm{1}$, we apply binomial identities of type $(a-b)a = \frac{1}{2} \left( a^2 - b^2 + (a-b)^2\right)$ and telescope sums, resulting in
\begin{align}
\nonumber
 T_\mathrm{1} &= \frac{1}{2} \left[\weightedsquare{\incu{M}}{\solidstabilization} - \weightedsquare{\incu{m}}{\solidstabilization} + \sum_{k=m+1}^M \weightedsquare{\incu{k} - \incu{k-1}}{\solidstabilization} \right]\\
 \nonumber
 &\qquad
 +
 \frac{\Delta t}{2} \left[\weightedsquare{\incv{M}}{\fluidstabilization} - \weightedsquare{\incv{m}}{\fluidstabilization} + \sum_{k=m+1}^M \weightedsquare{\incv{k} - \incv{k-1}}{\fluidstabilization} \right] \\
 &\qquad +
 \frac{\Delta t}{2} \left[\weightedsquare{\incp{M}}{\pressurestabilization} - \weightedsquare{\incp{m}}{\pressurestabilization} + \sum_{k=m+1}^M \weightedsquare{\incp{k} - \incp{k-1}}{\pressurestabilization} \right].
\label{fp-s-analysis-3}
\end{align}
(employing notation (N1)). For the coupling term $T_2$ we apply summation by parts, leading to
\begin{align}
\label{fp-s-analysis-4}
 T_2 &= 
\llangle \incp{M}, \DIV \left( (1-\porosity) \incu{M} \right) \rrangle
 -
 \llangle \incp{m}, \DIV \left((1-\porosity) \incu{m} \right) \rrangle\\
 \nonumber
 &\quad 
 -
 \sum_{k=m+1}^M   \llangle \incp{k} - \incp{k-1}, \DIV \left((1-\porosity) \incu{k-1} \right) \rrangle.
\end{align}
For the coupling term $T_3$, simple expansion and reformulation, aiming at constructing quadratic terms present on the left hand side of~\eqref{fp-s-analysis-2}, and gathering those, respectively, results in
\begin{align}
\label{fp-s-analysis-5}
 T_3 &= - \Delta t \sum_{k=m+1}^M \left\|\incv{k} - \frac{\incu{k} + \incu{k-1}}{2 \Delta t}\right\|^2_{\porosity^2\permeability^{-1}} - \frac{1}{2\Delta t} \left\| \incu{M} \right\|^2_{\porosity^2\permeability^{-1}} + \frac{1}{2\Delta t} \left\| \incu{m} \right\|^2_{\porosity^2\permeability^{-1}}\\
 \nonumber 
 &\qquad 
 - \frac{1}{4\Delta t} \sum_{k=m+1}^M \left\| \incu{k} - \incu{k-1} \right\|^2_{\porosity^2\permeability^{-1}}.
\end{align}
Inserting~\eqref{fp-s-analysis-3}--\eqref{fp-s-analysis-5} into~\eqref{fp-s-analysis-2} and re-ordering terms, yields
\begin{align}
\nonumber
&\sum_{k=m+1}^M \left[ \frac{1}{\Delta t^2} \| \incu{k} \|_{\rhos(1-\porosity)}^2 
+
\left\| \eps{\incu{k}} \right\|_{\mathbb{C}}^2
+
\| \incv{k}\|_{\rhof\porosity}^2 +  \Delta t\,  \left\| \eps{\incv{k}} \right\|_{\porosity 2\mu_\mathrm{f}}^2 
+
\left\| \incp{k} \right\|_{\frac{(1-\porosity)^2}{\kappa_\mathrm{s}}}^2 \right] \\
\nonumber
&\qquad +
\frac{1}{2} \left[\weightedsquare{\incu{M}}{\solidstabilization} + \sum_{k=m+1}^M \weightedsquare{\incu{k} - \incu{k-1}}{\solidstabilization} \right] + 
\frac{\Delta t}{2} \left[\weightedsquare{\incv{M}}{\fluidstabilization} + \sum_{k=m+1}^M \weightedsquare{\incv{k} - \incv{k-1}}{\fluidstabilization} \right]\\
\nonumber
&\qquad + \frac{\Delta t}{2} \left[\weightedsquare{\incp{M}}{\pressurestabilization} + \sum_{k=m+1}^M \weightedsquare{\incp{k} - \incp{k-1}}{\pressurestabilization} \right] +
\Delta t \sum_{k=m+1}^M \left\|\incv{k} - \frac{\incu{k} + \incu{k-1}}{2 \Delta t}\right\|^2_{\porosity^2\permeability^{-1}} 
\\
 \nonumber
 &\qquad
 + \frac{1}{2\Delta t} \left\| \incu{M} \right\|^2_{\porosity^2\permeability^{-1}}
 +
 \frac{1}{4\Delta t} \sum_{k=m+1}^M \left\| \incu{k} - \incu{k-1} \right\|^2_{\porosity^2\permeability^{-1}}  -
\underbrace{\llangle \incp{M}, \DIV \left( (1-\porosity) \incu{M} \right) \rrangle}_{=:T_4}
\\
\nonumber
&\qquad =
 \underbrace{-\llangle \incp{m}, \DIV \left((1-\porosity) \incu{m} \right) \rrangle}_{=:T_\mathrm{5a}}
 \underbrace{-\sum_{k=m+1}^M   \llangle \incp{k} - \incp{k-1}, \DIV \left((1-\porosity) \incu{k-1} \right) \rrangle}_{=:T_\mathrm{5b}} \\
&\qquad +
\frac{1}{2\Delta t} \left\| \incu{m} \right\|^2_{\porosity^2\permeability^{-1}}+ \frac{1}{2} \weightedsquare{\incu{m}}{\solidstabilization}
+
 \frac{\Delta t}{2} \weightedsquare{\incv{m}}{\fluidstabilization}
 +
 \frac{\Delta t}{2} \weightedsquare{\incp{m}}{\pressurestabilization}.
\label{fp-s-analysis-6}
\end{align}
We discuss the coupling terms $T_4$, $T_\mathrm{5a}$ and $T_\mathrm{5b}$ separately in the two following steps.
\paragraph{Revisiting the increment equation for the solid for the last iteration} The coupling term $T_4$ combined with terms in~\eqref{fp-s-analysis-6}, involving $\incu{M}$, constitutes a positive contribution. Indeed, (i) revisiting~\eqref{fp-s-analysis-1a} tested with $\testu=\incu{M}$, (ii) suitable expansion and reformulation, and ultimately (iii) discarding some positive terms and employing the definition of $\hatsolidstabilization$, yields for all terms of~\eqref{fp-s-analysis-6} involving $\incu{M}$
  \begin{align}
  \nonumber
   & \frac{1}{\Delta t^2} \| \incu{M} \|_{\rhos(1-\porosity)}^2 
    +
   \left\| \eps{\incu{M}} \right\|_{\mathbb{C}}^2
   +
   \frac{1}{2} \weightedsquare{\incu{M}}{\solidstabilization}
   + 
   \frac{1}{2} \weightedsquare{\incu{M} - \incu{M-1}}{\solidstabilization}
   -
   T_4\\
  \nonumber
   &\qquad + \Delta t \left\|\incv{M} - \frac{\incu{M} + \incu{M-1}}{2 \Delta t}\right\|^2_{\porosity^2\permeability^{-1}}
   +
   \frac{1}{2\Delta t} \left\| \incu{M} \right\|^2_{\porosity^2\permeability^{-1}}
   +
   \frac{1}{4\Delta t} \left\| \incu{M} - \incu{M-1} \right\|^2_{\porosity^2\permeability^{-1}}\\
  \nonumber
   &\quad \underset{(i)}{=}
   \frac{1}{2} \weightedsquare{\incu{M-1}}{\solidstabilization}
   +
   \Delta t \llangle \porosity^2 \permeability^{-1} \left( \incv{M} - \frac{\incu{M}}{\Delta t} \right), \frac{\incu{M}}{\Delta t} \rrangle\\
  \nonumber
   &\qquad +
   \Delta t \left\|\incv{M} - \frac{\incu{M} + \incu{M-1}}{2 \Delta t}\right\|^2_{\porosity^2\permeability^{-1}}
   +
   \frac{1}{2\Delta t} \left\| \incu{M} \right\|^2_{\porosity^2\permeability^{-1}}
   +
   \frac{1}{4\Delta t} \left\| \incu{M} - \incu{M-1} \right\|^2_{\porosity^2\permeability^{-1}}\\
  \nonumber
   &\quad 
   \underset{(ii)}{=}
   \frac{1}{2} \weightedsquare{\incu{M-1}}{\solidstabilization}
   +
   \frac{1}{2\Delta t} \left\| \incu{M-1} \right\|_{\porosity^2\permeability^{-1}}^2 + \Delta t \left\| \incv{M} - \frac{\incu{M-1}}{2\Delta t} \right\|_{\porosity^2\permeability^{-1}}^2
   +\frac{\Delta t}{2} \left\| \incv{M} \right\|_{\porosity^2\permeability^{-1}}^2\\
   &\quad 
   \underset{(iii)}{\geq} 
   \frac{1}{2} \weightedsquare{\incu{M-1}}{\hatsolidstabilization}.
   \label{fp-s-analysis-7a}
  \end{align}
  \paragraph{Bounding coupling terms $T_\mathrm{5a}$ and $T_\mathrm{5b}$} 
  We employ (i) the product rule, (ii) the definitions of $K_\mathrm{dr,\porosity,min}$ and $C_\mathrm{Korn,1}$, cf.~\eqref{eq:bulk-modulus} and~\eqref{eq:korn1}, and (iii) the Cauchy-Schwarz and Young's inequalities. After all, for any $\delta_1>0$, we bound $T_\mathrm{5a}$
  \begin{align}
  \nonumber
   T_\mathrm{5a}
   &\underset{(i)}{=}
   \llangle \incp{m}, \GRAD \porosity \cdot \incu{m} \rrangle 
   -
   \llangle \incp{m}, (1-\porosity) \DIV \incu{m} \rrangle\\
   \nonumber
   &\underset{(ii)}{\leq}
   \left\|\incp{m} \right\| \left( C_\mathrm{Korn,1}^{1/2} + K_\mathrm{dr,\porosity,min}^{-1/2} \right) \left\| \eps{\incu{m}} \right\|_{\mathbb{C}}\\
   &\underset{(iii)}{\leq} 
   \frac{\delta_1}{2} \left\| \eps{\incu{m}} \right\|_{\mathbb{C}}^2
   +
   \frac{1}{2\delta_1}\left( C_\mathrm{Korn,1}^{1/2} + K_\mathrm{dr,\porosity,min}^{-1/2} \right)^2   
   \left\|\incp{m} \right\|^2.
   \label{fp-s-analysis-7b}
  \end{align}
  Similarly for $T_\mathrm{5b}$, we obtain for any $\delta_2>0$
  \begin{align}
  \label{fp-s-analysis-7c}
   T_\mathrm{5b}
   \underset{(i)-(iii)}{\leq}
   \frac{\delta_2}{2} \sum_{k=m}^{M-1} \left\| \eps{\incu{k}} \right\|_{\mathbb{C}}^2
   +
   \frac{1}{2\delta_2} \left( C_\mathrm{Korn,1}^{1/2} + K_\mathrm{dr,\porosity,min}{-1/2} \right)^2 
   \sum_{k=m+1}^M\left\|\incp{k} - \incp{k-1} \right\|^2.
  \end{align}
  
  \paragraph{Conclusion of relative stability} Inserting~\eqref{fp-s-analysis-7a}--\eqref{fp-s-analysis-7c} into~\eqref{fp-s-analysis-6} and employing $\hatpressurestabilization$, yields
  \begin{align}
    \nonumber
    &\sum_{k=m+1}^{M-1} \left[ \frac{1}{\Delta t^2} \| \incu{k} \|_{\rhos(1-\porosity)}^2 
    +
    \left(1 - \frac{\delta_2}{2} \right) \left\| \eps{\incu{k}} \right\|_{\mathbb{C}}^2 \right]\\
    \nonumber
    &\qquad 
    +
    \sum_{k=m+1}^{M} \left[
    \| \incv{k}\|_{\rhof\porosity}^2 +  \Delta t\,  \left\| \eps{\incv{k}} \right\|_{\porosity 2\mu_\mathrm{f}}^2 
    +
    \left\| \incp{k} \right\|_{\frac{(1-\porosity)^2}{\kappa_\mathrm{s}}}^2 \right] \\
    \nonumber
    &\qquad +
    \frac{1}{2} \weightedsquare{\incu{M-1}}{\hatsolidstabilization}
    +
    \frac{1}{2}\sum_{k=m+1}^{M-1} \weightedsquare{\incu{k} - \incu{k-1}}{\hatsolidstabilization}
    +
    \Delta t \sum_{k=m+1}^{M-1} \left\|\incv{k} - \frac{\incu{k} + \incu{k-1}}{2 \Delta t}\right\|^2_{\porosity^2\permeability^{-1}}
    \\
    \nonumber
    &\qquad +
    \frac{\Delta t}{2} \left[\weightedsquare{\incv{M}}{\fluidstabilization} + \sum_{k=m+1}^M \weightedsquare{\incv{k} - \incv{k-1}}{\fluidstabilization} \right]\\
    \nonumber
    &\qquad 
    +
    \frac{\Delta t}{2} \left[\weightedsquare{\incp{M}}{\pressurestabilization} + \sum_{k=m+1}^M 
    \weightedsquare{\incp{k} - \incp{k-1}}{\pressurestabilization - \frac{1}{\delta_2\, \Delta t}\left( C_\mathrm{Korn,1}^{1/2} + K_\mathrm{dr,\porosity,min}^{-1/2} \right)^2}
    \right]\\
    &\quad \leq
    \frac{1}{2} \weightedsquare{\incu{m}}{\hatsolidstabilization}
    +
    \frac{\delta_1 + \delta_2}{2} \left\| \eps{\incu{m}} \right\|_{\mathbb{C}}^2 + \frac{\Delta t}{2} \weightedsquare{\incv{m}}{\fluidstabilization} +
   \frac{\Delta t}{2} 
    \weightedsquare{\incp{m}}{\hatpressurestabilization}. 
    \label{fp-s-analysis-8}
    \end{align}
    Finally, after choosing $\solidstabilization$, $\fluidstabilization$ and $\pressurestabilization$ satisfying~\eqref{fp-s-stabilization-parameters} (in particular translating to $\hatsolidstabilization\succeq\bm{0}$), and dropping several positive terms in~\eqref{fp-s-analysis-8}, we obtain the stability result
  \begin{align*}
    &\sum_{k=m+1}^{M-1} \left[ \frac{1}{\Delta t^2} \| \incu{k} \|_{\rhos(1-\porosity)}^2 
    +
    \left(1 - \frac{\delta_2}{2} \right) \left\| \eps{\incu{k}} \right\|_{\mathbb{C}}^2 
    +
    \| \incv{k}\|_{\rhof\porosity}^2 +  \Delta t\,  \left\| \eps{\incv{k}} \right\|_{\porosity 2\mu_\mathrm{f}}^2 
    +
    \left\| \incp{k} \right\|_{\frac{(1-\porosity)^2}{\kappa_\mathrm{s}}}^2 \right] \\
    &\quad \leq
    \frac{1}{2} \left\| \incu{m} \right\|_{\hatsolidstabilization}^2
    +
    \frac{\delta_1 + \delta_2}{2} \left\| \eps{\incu{m}} \right\|_{\mathbb{C}}^2
    +
    \frac{\Delta t}{2} \left\| \incv{m} \right\|_{\fluidstabilization}^2
    + 
   \frac{\Delta t}{2}  
    \left\| \incp{m} \right\|_{\hatpressurestabilization}^2.
    \end{align*}
   After all, relative stability in the sense of~\eqref{convergence-criterion-stability} can be deduced for any choice for $\delta_1>0$ and $\delta_2\in(0,2)$, since $m$ and $M$ have been chosen arbitrary. By this the assertion follows. 
\end{proof}

\begin{remark}[Incompressible media]
 We note that in contrast to the alternating minimization split~\eqref{undrained-split-1}--\eqref{undrained-split-2}, the diagonally $L^2$--stabilized two-way split~\eqref{fp-s-1}--\eqref{fp-s-2} remains well defined in the extreme case of (quasi-)incompressible solid material, i.e., $\frac{(1-\porosity)^2}{\kappa_\mathrm{s}}=0$. According to the above theory, convergence is not guaranteed anymore, yet still may be possible in practice, see also examples in Section~\ref{section:numerical-tests}.
\end{remark}

We close this section with suggesting a practical set of stabilization parameters guided by the previous convergence analysis. We emphasize that one could optimize the effective stability constant in~\eqref{fp-s-stability} wrt.\ $\delta_1,\ \delta_2$, $\solidstabilization,\ \fluidstabilization,\ \pressurestabilization$; however, theoretical optimality does not necessarily result in practical optimality, cf.~\cite{Storvik2019} for an applicable discussion.

\begin{remark}[A practical set of stabilization parameters]
We assume heterogeneities of the porosity are not crucial and pretend the porosity is constant. Then it is $C_\mathrm{Korn,1}=0$ and $K_\mathrm{dr,\porosity,min}^{-1}=\frac{(1-\porosity)^2}{K_\mathrm{dr}}$, where $K_\mathrm{dr}=K_\mathrm{dr,0,min}$ denotes the standard drained bulk modulus. Moreover, we choose the values $\delta_1=\delta_2=1$ in order to balance similar terms on both sides of~\eqref{fp-s-stability} and follow the suggestion of the stability property to choose the stabilization parameters as ``small'' as possible. This results in the set
 \begin{align*}
  \solidstabilization = - \frac{\porosity^2 \permeability^{-1}}{2\Delta t},\qquad\qquad
  \fluidstabilization = \bm{0},\qquad \qquad 
  \pressurestabilization = \frac{(1-\porosity)^2}{K_\mathrm{dr}\,\Delta t},
 \end{align*}
 which leads to destabilization of the momentum equation of the solid. However, we also highlight that merely utilizing pressure stabilization and setting $\solidstabilization=\fluidstabilization=\bm{0}$ does also result in guaranteed convergence, in the style of the fixed-stress split for the quasi-static Biot equations.
\end{remark}
 
\section{Numerical tests for the convergence of the proposed splitting schemes}\label{section:numerical-tests}

The aim of this section is to assess the performance of the proposed splitting schemes, the alternating minimization split, cf.\ Section~\ref{section:undrained-split}, and the diagonally $L^2$-stabilized two-way split, cf.\ Section~\ref{section:fixed-stress},
and to compare it with the theoretical convergence results in Section~\ref{section:analysis}. In particular, we consider three test cases and 
%
perform an extensive parametric study for various choices of model parameters and stabilization values based on the above analyses, in addition to similar ad-hoc choices motivated by the analyses or experience of the closely related splitting schemes for the Biot equations.  

As test problems, we use two classic problems, the swelling~\cite{Burtschell2017effective, Burtschell201928, Vuong20151240} and footing~\cite{Adler2017,adler2017robust,Storvik2019} problems. In addition, we consider a perfusion-like problem as a reference for biomedical applications. We note that each problem is loaded on a different equation: the swelling on the fluid, the footing on the solid and the perfusion on the mass balance.

We first present a sensitivity study with respect to the physical parameters for both alternating minimization and $L^2$--stabilized splits independently based on the swelling test. Then, we provide a detailed comparison between both methods in all the described test problems in combination with Anderson acceleration. At the end of this section, we also compare the performance of the split scheme that results most effective, with a monolithic solution approach for the linearized problem, which may be considered to be the \emph{gold standard} solution strategy. This final test sheds light on the competitiveness of the proposed schemes when used for realistic scenarios.

All numerical examples have been performed using the FEniCS project~\cite{logg2012automated,alnaes2015fenics}, and convergence is measured in terms of the relative residual (for larger certainty absolute residuals are not considered).

\subsection{Definition of the test cases\label{section:numerics-test-case}}

\paragraph{The swelling test} This test consists of a 2D slab $\Omega=(0, L)^2,\ L=10^{-2}$, in absence of volume forces and simulated in the time interval $(0,1)$, with time step $\Delta t=0.1$. It is subject to an inflow $\porosity\left(2\mu_\mathrm{f} \eps\fluidvelocity - p\tensor I \right)\boldsymbol n=-p_\text{ext}\boldsymbol n, p_\text{ext}(t)=10^3(1-\exp(4t^2))$ on the left and null stress on the right, whereas above and below it uses a no-slip boundary condition $\fluidvelocity=0$. The boundary conditions for the solid are sliding on the bottom and left sides, whereas the rest of the boundary is of null traction type (see Figure \ref{fig:numerical-tests-bcs} (a)). We note that these conditions are not physical because the fluid boundary pressure should act as force on the solid as well, but we keep the proposed scenario to have this test being loaded only on the fluid equation. We have indeed tested this and observed that it presents no impact on the following study.
The following default parameters (from~\cite{Burtschell201928}) are used (unless otherwise specified): $\rho_\mathrm{f}=\rho_\mathrm{s} = 1000, \mu_\mathrm{f} = 0.035, \lambda_\mathrm{s} = 711, \mu_\mathrm{s} = 4066, \kappa_\mathrm{s} = 10^3, \permeability = 10^{-7}\tensor I, \porosity = 0.1$, all in SI units; in addition $\Omega$ is discretized using 10 elements per side. 
Let us denote by $\boldsymbol X_h^k(\Omega)$ the Lagrangian $k$-th order finite element space defined on a quasi-uniform mesh of $\Omega$ of characteristic size $h$.
Unless otherwise specified, the default finite element spaces used are: first order Lagrangian elements for the solid, $\boldsymbol V_h = \boldsymbol X_h^1(\Omega)$ and Taylor-Hood elements for the fluid-pressure system, $\boldsymbol W_h \times Q_h = \boldsymbol X_h^2(\Omega) \times X_h^1(\Omega)$, which satisfy the weighted inf-sup condition  partially but are more useful in practice \cite{BARNAFI2020}. We name this choice of elements with the shorthand notation P1/P2/P1.
Finally, a relative tolerance of $10^{-8}$ was used with respect to the $\ell^\infty$ norm of the residual, where all sub-problems are solved using GMRES with a relative tolerance of $10^{-8}$ as well, preconditioned with with an incomplete LU factorization with 3 levels of depth (ILU(3)~\cite{saad2003iterative}).


\paragraph{The footing test} This test (from~\cite{adler2017robust}) also consists of a 2D slab $\Omega=(0, L)^2,\ L=64$, simulated in the time interval $(0,1)$, with time step $\Delta t=0.01$ in absence of volume forces where half of the boundary on top $\Gamma_\text{foot} =  (16, 48) \times \{64\}$ is subject to an increasing load. More precisely, the fluid phase is subject to a \emph{no-slip} condition on $\Gamma_\text{foot}$ and null pressure in $\partial\Omega\setminus \Gamma_\text{foot}$. The boundary conditions for the solid are given by an increasing load $\boldsymbol t(\boldsymbol x, t) = (0, -10^5 t)$ on $\Gamma_\text{foot}$, homogeneous Dirichlet conditions on the bottom $\boldsymbol{u}_\mathrm{s} = \boldsymbol 0$ and null Neumann conditions everywhere else (see Figure \ref{fig:numerical-tests-bcs} (b)). The parameters used are given by: $\rho_\mathrm{f}=1000, \rho_\mathrm{s} = 500, \mu_\mathrm{f} = 10^{-3}, E = 3\cdot10^4, \nu = 0.2, \lambda_\mathrm{s} =  E \nu / ((1 + \nu) (1 - 2  \nu)), \mu_\mathrm{s} = E / (2  (1 + \nu)), \kappa_\mathrm{s} = 10^6, \permeability = 10^{-7}\tensor I, \porosity = 10^{-3}$, all in SI units, discretized using 10 elements per side, with two simple refinements performed near the footing boundary. The finite element spaces used are the ones of the swelling test, namely P1/P2/P1. Finally, a relative tolerance of $10^{-6}$ is used with respect to the $\ell^\infty$ norm of the residual.

\paragraph{The perfusion test} This test also consists of a 2D slab $\Omega=(0, L)^2,\ L=0.01$ simulated in the time interval $(0,1)$, with time step $\Delta t=0.1$. Both fluid and solid phases are subject to homogeneous Dirichlet boundary conditions on the left and homogeneous Neumann conditions elsewhere (see Figure \ref{fig:numerical-tests-bcs} (c)). We set the scalar source term $\theta = 500$, and the problem parameters are given by: $\rho_\mathrm{f}=1000, \rho_\mathrm{s} = 1000, \mu_\mathrm{f} = 0.03, E = 3\cdot10^4, \lambda_\mathrm{s} =  5\cdot10^4, R = \sqrt{E^2 + 9\lambda_\mathrm{s}^2 + 2E\lambda_\mathrm{s}}, \mu_\mathrm{s} = 0.25\,(E - 3\lambda_\mathrm{s} + R), \kappa_\mathrm{s} = 10^6, \permeability = 10^{-9}\tensor I, \porosity = 0.05$, all referring to in SI units. These mechanical parameters are obtained from \cite{shaw2013mechanical}, the remaining ones from \cite{michler2013}.  The domain is discretized using 10 elements per side, and the finite element spaces used are P1/P2/P1. A relative tolerance of $10^{-8}$ is used with respect to the $\ell^\infty$ norm of the residual. 
%
%
%
%
\begin{figure}[!ht]
    \centering
    \begin{subfigure}[c]{0.45\textwidth}
    \hspace{0.65cm}
    \includegraphics[height=4.5cm]{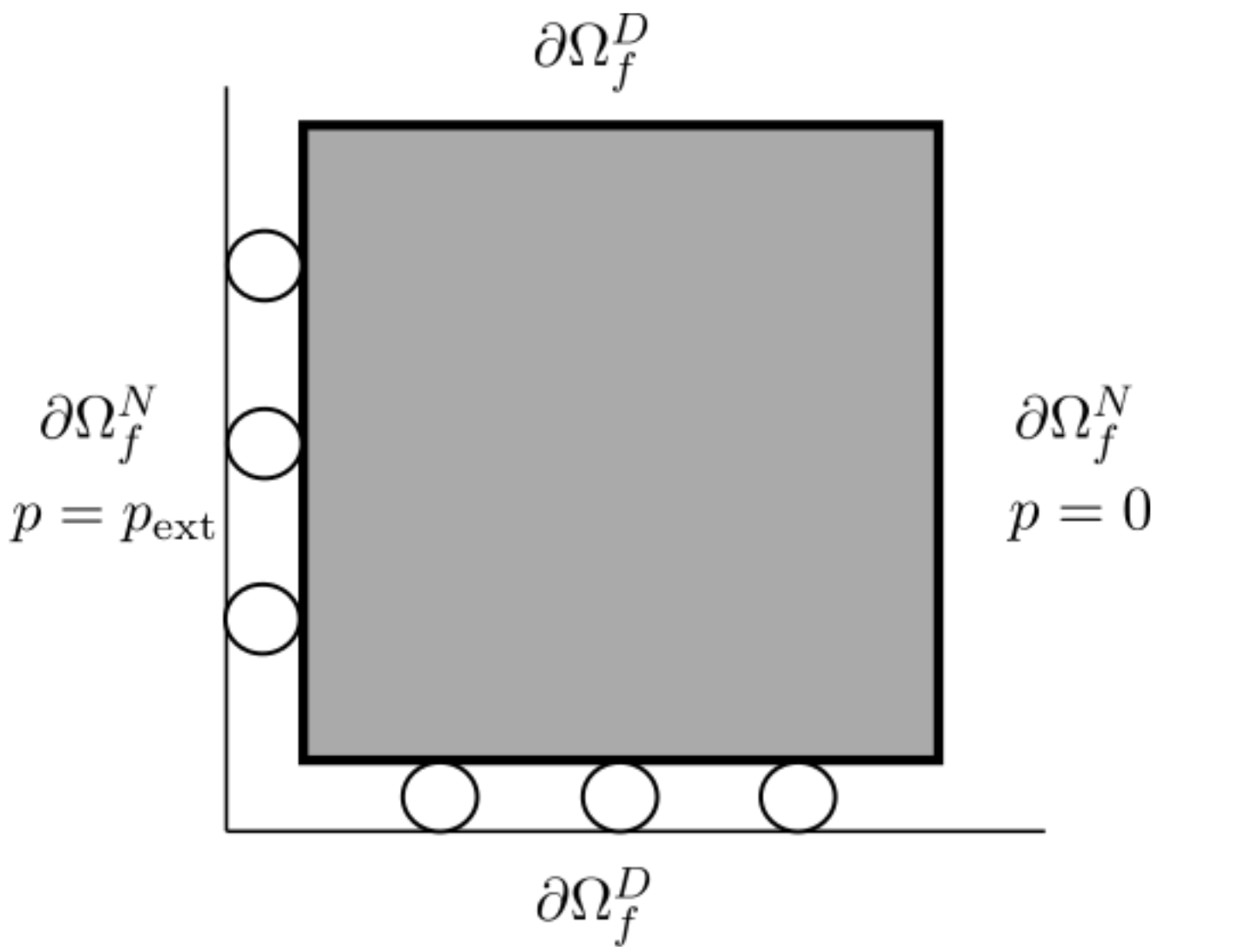}
    \caption{Swelling BC.}
    \end{subfigure}
    \begin{subfigure}[c]{0.45\textwidth}
    \hspace{1.4cm}
    \includegraphics[height=4.5cm]{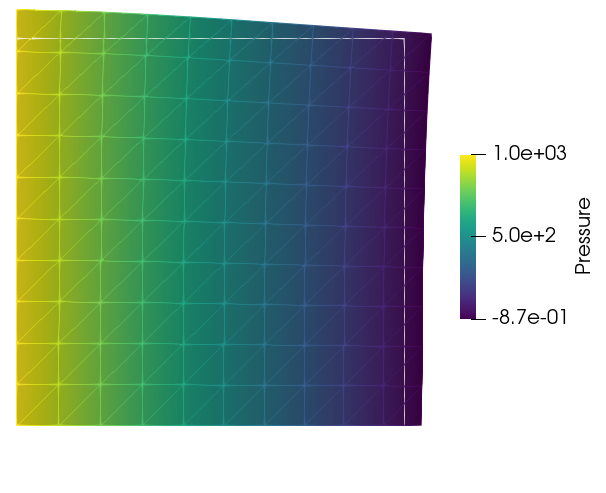}
    \caption{Swelling solution at $t=1s$.}
    \end{subfigure}
    
    \begin{subfigure}[c]{0.45\textwidth}
    \hspace{0.5cm}
    \includegraphics[height=4.5cm, trim=-1.9cm 0 0 0]{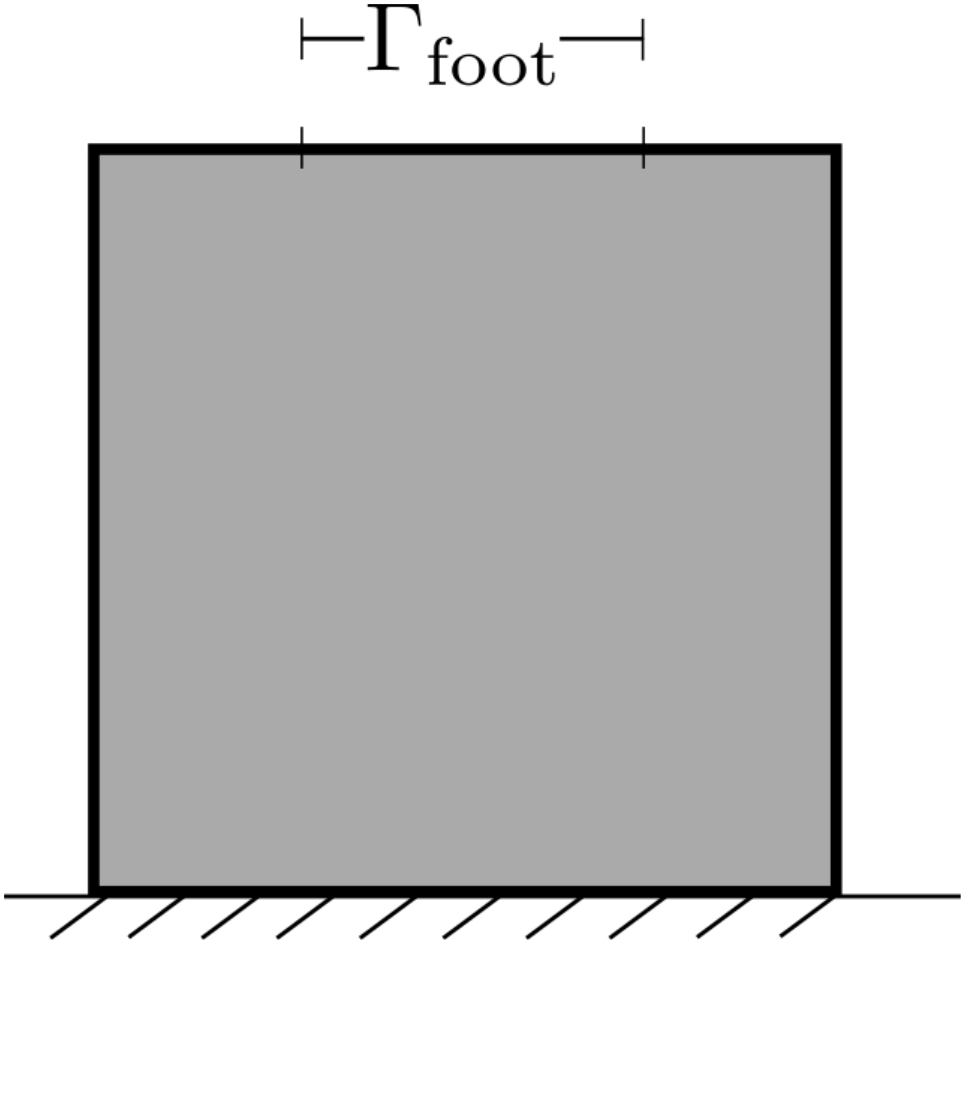} 
    \caption{Footing BC.}
    \end{subfigure}
    \begin{subfigure}[c]{0.45\textwidth}
    \hspace{1.5cm}
    \includegraphics[height=4.5cm]{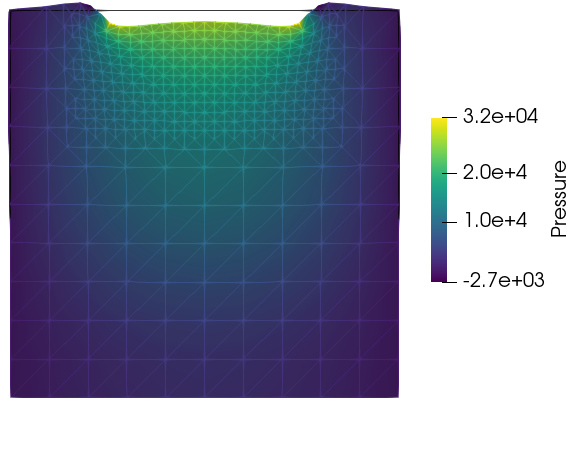}
    \caption{Footing solution at $t=1s$.}
    \end{subfigure}

    \begin{subfigure}[c]{0.45\textwidth}
    \hspace{0.75cm}
    \begin{overpic}[height=4cm, trim=0.5cm 0 0 0]{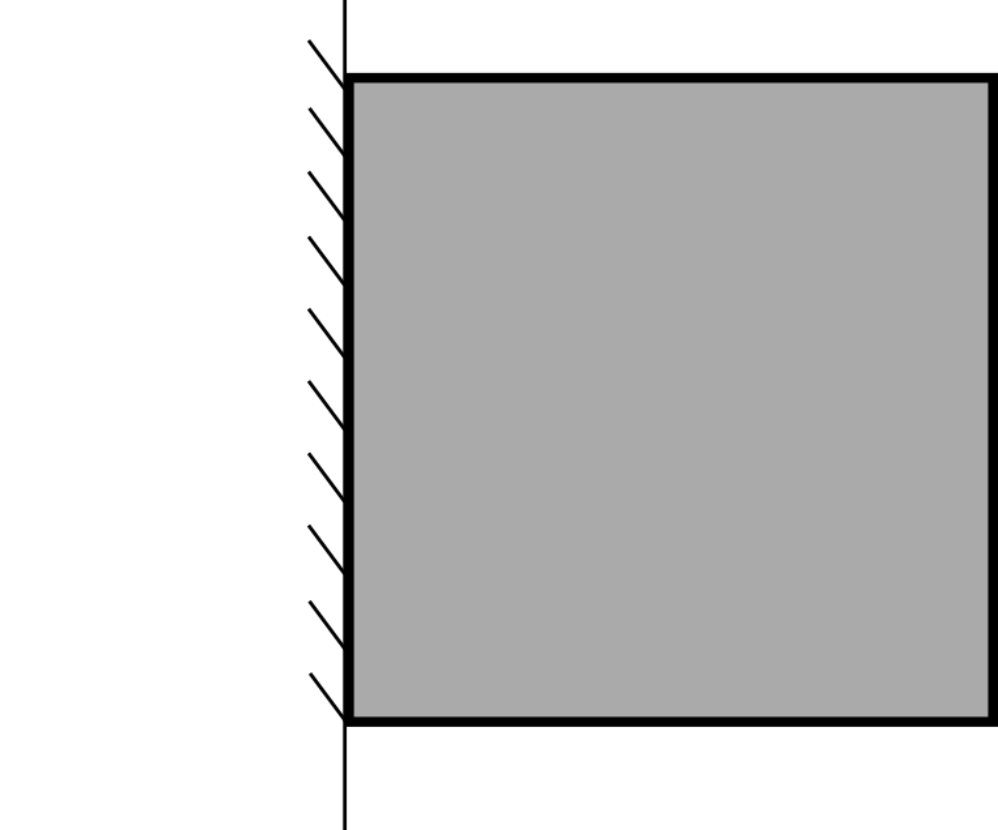}
     \put(-2.8,50){$\displacement = \bm{0}$}
     \put(-2,40){$\fluidvelocity = \bm{0}$}
     \end{overpic}
    \caption{Perfusion BC.}
    \end{subfigure}
    \begin{subfigure}[c]{0.45\textwidth}
    \hspace{1.5cm}
    \includegraphics[height=4.5cm]{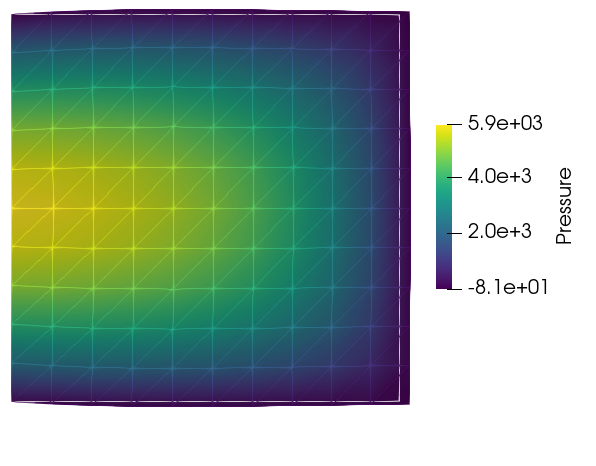}
    \caption{Perfusion solution at $t=1s$.}
    \end{subfigure}
    \caption{Boundary conditions used in the numerical tests and the corresponding solution.}
    \label{fig:numerical-tests-bcs}
\end{figure}

\subsection{Anderson acceleration}\label{section:anderson-acceleration}
One key aspect of both proposed schemes is that they can be interpreted as fixed point iterations. Although they feature in general lower convergence rates than  Newton methods, they have acquired higher interest recently, also due to the development of acceleration schemes. In particular, we focus on the Anderson acceleration, which can be interpreted as a multisecant scheme, or as a preconditioned GMRES iterative method \cite{walker2011anderson}. 
As shown later on in Tables~\ref{table:iterative:swelling-comparison},~\ref{table:iterative:footing-comparison},~\ref{table:iterative:perfusion-comparison},
acceleration techniques greatly improve the performance of the proposed split schemes, by increasing their robustness with respect to varying loading conditions and significantly reducing the iteration count. In practice, using Anderson acceleration is a necessary choice to effectively use the described split schemes in demanding scenarios.

In general, consider a vector-valued function $\vec g:\mathbb{R}^N\rightarrow \mathbb{R}^N$ and the sequence     
\begin{align*}
\vec x_{k+1} = \vec g(\vec x_k).
\end{align*}
By defining $\vec f_k =  \vec g(\vec x_k) - \vec x_{k}$, Anderson acceleration of order $m$, abbreviated by AA($m$), is given as follows: For iteration $k$, set $m_k = \min\{m, k\}$ and $\mat F_k = (\vec f_{k - m_k}, ...,\vec f_k)$. Compute $\vec \alpha^{k} = (\alpha_0^k, ..., \alpha_{m_k}^k)$ that minimizes
    \begin{equation}\label{eq:iterative:anderson-original}
        \min_{\alpha = (\alpha_0, ..., \alpha_{m_k})}  \quad \| \mat F\vec \alpha \|_2 
        \quad \text{s.t.}\quad 
        \sum_{i=0}^{m_k}\alpha_i = 1,
    \end{equation}
and then compute the next element as 
\begin{align*}
    \vec x_{k+1} = \sum_{i=0}^{m_k} \alpha^k_i \vec g(\vec x_{k - m_k + i}).
\end{align*} 
The order $m$ of the scheme is usually referred to as depth, due to the use of $m$ previous iterations. We implement this method by recasting \eqref{eq:iterative:anderson-original} as an unconstrained least-squares problem, and then invert its optimality conditions
using the QR factorization to avoid possible ill-conditioning of the normal equations \cite{saad2003iterative}.

\subsection{Numerical tests for the alternating minimization split\label{section:numerics-undrained}}

In this section we present three numerical tests on the alternating minimization split (named \emph{Alt--min} in the tables), with the aim of verifying the robustness of the scheme with respect to the parameters $N=\frac{\kappa_\mathrm{s}}{(1-\porosity)^2}, \permeability$ and highly oscillatory porosities $\porosity$. As test case we adopt the swelling test described above. We consider three varying parameters
\begin{align*}
     \kappa_\mathrm{s} \in \{ 10^k \}_{k=2}^5; \quad
     \kappa_\mathrm{f} \in \{ 10^{-k} \}_{k=9}^{12}; \quad
     \porosity = 0.1 + 0.5\,\sin^2(\ell \pi x/L), \,\ell=2,...,8;
\end{align*}
 where $L=10^{-2}$ is the side length, and the permeability is treated as a scalar for simplicity, namely $\permeability = \kappa_\mathrm{f} \boldsymbol I$. For each parameter aside of default parameters otherwise, we present the average number of splitting iterations throughout the simulation required for convergence in Table~\ref{table:iterative:undrained-sensitivity}.
 We observe that the performance of the alternating minimization split is particularly sensitive to the bulk modulus $\kappa_\mathrm{s}$, and small permeabilities make the problem much more difficult to solve. Instead, the dependence on oscillating porosity is moderate. The results are in accordance to Theorem \ref{theorem:convergence-undrained}.

 \begin{table}[!ht]
     \centering
     \begin{subtable}[b]{0.32\textwidth}
     \centering
         \begin{tabular}{ccc}\toprule
              $\kappa_\mathrm{s}$ & &\# avg. iters. \\
              \cmidrule{1-1}\cmidrule{3-3}
              $10^2$ && 8.55 \\
              $10^3$     && 15.91 \\
              $10^4$     && 64.09  \\
              $10^5$     && --  \\
              \bottomrule
         \end{tabular}
         \caption{Bulk modulus.}
     \end{subtable}
     \begin{subtable}[b]{0.32\textwidth}
     \centering
          \begin{tabular}{ccc}\toprule
              $\kappa_\mathrm{f}$ && \# avg. iters. \\
              \cmidrule{1-1}\cmidrule{3-3}
              $10^{-9}$ && 17.64 \\
              $10^{-10}$ && 73.72 \\
              $10^{-11}$ && 399.96 \\
              $10^{-12}$ && -- \\
              \bottomrule
         \end{tabular}
         \caption{Permeability.}
     \end{subtable}
     \begin{subtable}[b]{0.32\textwidth}
         \centering
         \begin{tabular}{ccc}\toprule
              $\ell$ && \# avg. iters.  \\
              \cmidrule{1-1}\cmidrule{3-3}
              2  && 23.91 \\
              4  && 54.64\\
              6  && 103.09\\
              8  && 170.64\\ 
              \bottomrule
         \end{tabular}     
         \caption{Porosity $\porosity = \porosity(\ell)$.}
     \end{subtable}
     \caption{Alt--min for the swelling test: Average iteration count for varying (a) Bulk modulus (b) Permeability and (c) Porosity. Non-convergence denoted with -- after 200 iterations for bulk modulus and 500 iterations for permeability.}
     \label{table:iterative:undrained-sensitivity}
 \end{table}
 
\subsection{Numerical tests for the diagonally \texorpdfstring{$L^2$}{L2}--stabilized split}

In this section, we study the sensitivity of the performance of the diagonally $L^2$--stabilized split (named $L^2S$ in the tables) with respect to different combinations of physical parameters. Precisely, we use the swelling test with default coefficients, and we vary the following ones
\begin{align*}
\kappa_\mathrm{s} \in\{10^k\}_{k=2}^{8};
\quad
\kappa_\mathrm{f} \in \{10^{-k}\}_{k=7}^{12};
\quad
\rhos=\rhof \in \{10^k\}_{k=2}^8.
\end{align*}
Additional tests address the influence of the ratio between the elasticity and the permeability. For this, we fix the permeability and increase the drained bulk modulus $K_\mathrm{dr} = \lambda + \frac{2\mu}{d},\ d=2$ by scaling both Lam\'e parameters by the same factor. 

The analysis in Section \ref{fp-s-split} yields the interesting fact that the solid momentum equation can be destabilized. Therefore, we compare different stabilization parameters, also ones excluded by the theory in order to investigate the theoretically suggested parameter ranges. In particular, we apply the $L^2$ stabilized two-way split~\eqref{fp-s-1}--\eqref{fp-s-2} using stabilization parameters of type
\begin{equation*}
 \solidstabilization = \bar{\beta}_{\mathrm{s}} \frac{\porosity^2\permeability^{-1}}{\Delta t},
 \qquad
 \fluidstabilization = \bar{\beta}_{\mathrm{f}} \porosity^2 \permeability^{-1},
 \qquad
 \pressurestabilization = \bar{\beta}_{\mathrm{p}} \frac{(1-\porosity)^2}{\Delta t \, K_\mathrm{dr}}
\end{equation*}
with different scaling factors $\bar{\beta}_{\mathrm{s}}, \bar{\beta}_{\mathrm{f}}, \bar{\beta}_{\mathrm{p}}$, from now on denoted by $L^2S_{\bar{\beta}_\mathrm{s},\bar{\beta}_\mathrm{f},\bar{\beta}_\mathrm{p}}$. Considered scaling factors are listed in Table~\ref{table:fp-s:variants}. Splitting iterations are terminated via the tolerance $\mathrm{tol}_\mathrm{res} = 10^{-8}$. As in the previous test, performance is measured in terms of the average number of splitting iterations throughout the entire simulation, with non-convergence established whenever a solver requires more than 200 iterations. 
 
\begin{table}[!ht]
\def\arraystretch{1.1}
\small
\centering
 \begin{tabular}{lcccll}
 \toprule
  ID &$\bar{\beta}_\mathrm{s}$  & $\bar{\beta}_\mathrm{f}$ & $\bar{\beta}_\mathrm{p}$ & Description & Covered by Thm.~\ref{theorem:l2s-convergence}\\ 
  \midrule
  $L^2S_{0,0,0}$ &0 & 0 & 0 & Unstabilized split & \xmark \\
  $L^2S_{0,0,1}$ &0 & 0 & 1 & $L^2S$ with fixed-stress-type $p$-stabilization & \cmark \\
  $L^2S_{-0.5,0,1}$&$-\frac{1}{2}$ & 0 & 1 & $L^2S$ with conservative $\displacement$-destabilization & \cmark \\
  $L^2S_{-1,0,1}$&$-1$ & 0 & 1 & $L^2S$ with aggressive $\displacement$-destabilization & \xmark \\ 
  \bottomrule
 \end{tabular}
 \caption{Considered stabilization settings in the context of the diagonally $L^2$--stabilized split. \label{table:fp-s:variants}}
\end{table}

Although the analysis, cf.\ Thm.~\ref{theorem:l2s-convergence}, does not reveal any dependence on the particular discretization, it is developed under the underlying assumption that the discrete problems are uniquely solvable. To investigate potential effects of stability of the function spaces onto the stability of the splitting, we consider progressively unstable approximation spaces, namely P1/P2/P1 and P1/P1/P1 elements for displacement, velocity and pressure, respectively.

\subsubsection{Dependence on solid bulk} 
In Table~\ref{tab:fp-s:bulk}, the iteration counts for varying $\kappa_\mathrm{s}$ are displayed. We observe that for P1/P1/P1 elements no set of stabilization parameters enables convergence for larger $\kappa_\mathrm{s}$; we note that for increasing $\kappa_\mathrm{s}$, the uniform stability of the fluid-pressure system is lost. In contrast, the use of P1/P2/P1 elements adds uniform stability to the discretization and finally also uniform robustness to any of the stablized splittings. For a non-dominating $\displacement-\fluidvelocity$ coupling, destabilization of the solid momentum equation does not make a big difference. 

\begin{table}[!ht]
\footnotesize
\centering
     \begin{tabular}{ccccccccccc}\toprule
     && \multicolumn{4}{c}{P1/P1/P1 elements}& & \multicolumn{4}{c}{P1/P2/P1 elements}\\
          $\kappa_\mathrm{s}$ & &
          $L^2S_{0,0,0}$ & 
          $L^2S_{0,0,1}$ & 
          $L^2S_{-0.5,0,1}$ &
          $L^2S_{-1,0,1}$ & &
          $L^2S_{0,0,0}$ & 
          $L^2S_{0,0,1}$ & 
          $L^2S_{-0.5,0,1}$ &
          $L^2S_{-1,0,1}$\\
          \cmidrule{1-1}\cmidrule{3-6}\cmidrule{8-11}
$10^{2} $ && 6.73  & 6.0   & 5.91   & {5.82}  &&  6.73  & 6.73 & {6.36}  & 6.36 \\
$10^{4} $ && {13.27} & 22.81 & 22.82  & 22.91 &&  13.18 & 7.0   & {6.73} & 6.91 \\
$10^{6} $ && --    & --     & --    & --             &&  14.0  & 7.09  & {6.82} & 7.0 \\
$10^{8} $ && --    & --     & --    & --             &&  14.09 & 7.09  & {6.82} & 7.0 \\
          \bottomrule
     \end{tabular}
     \caption{Average iteration count of the $L^2$--stabilized solvers for a varying $\kappa_\mathrm{s}$ in the swelling test. Non-convergence denoted by --.}
     \label{tab:fp-s:bulk}
 \end{table}

 \subsubsection{Dependence on permeability\label{section:l2s-analysis-swelling-permeability-study}} 
 In Table~\ref{tab:fp-s:permeability}, the iteration counts for varying $\permeability=\kappa_\mathrm{f} \bm{I}$ are displayed. Here, a maximal count of 500 splitting iterations is used for better understanding the dependence on the permeability. Lower permeability makes the problem more difficult to solve.
 The reasons for this are: (i) decreasing the permeability leads to ill-conditioning of the $\displacement$-$\fluidvelocity$ block; and (ii) for lower permeabilities the ellipticity of the $\displacement$-$\fluidvelocity$ block loses its dominance, and instead the $L^2$-type contribution has a much bigger influence.

 Destabilization of $\displacement$ seems to effectively address the first issue. In fact, it results in significantly improving the performance, compared to mere $p$-stabilization, which alone fails to lead to unconditional robustness. This, on the one hand, nicely verifies the theory in Thm.~\ref{theorem:l2s-convergence}. On the other hand, it indicates that suitable destabilization successfully imitates approximating the Schur complement of the $L^2$-type contribution of the $\displacement$-$\fluidvelocity$ block; the comparison of conservative and aggressive destabilization illustrates the potential gain but also sensitivity of destabilization. Since $L^2$-stabilization of the mass conservation equation does not address the $L^2$-type contribution of the $\displacement$-$\fluidvelocity$ block at all, unconditional robustness cannot be expected without an additional differently scaled stabilization approach, ultimately mitigating the second issue.
 
 Comparing the results for the P1/P1/P1 and P1/P2/P1 discretizations, we note that inf-sup stability in the fluid allows for a significant  improvement on the performance. Also, in contrast to the unstable case, destabilizing the $\displacement$ equations greatly improves performance.


 \begin{table}[!ht]
 \footnotesize
\centering
     \begin{tabular}{ccccccccccc}\toprule
     && \multicolumn{4}{c}{P1/P1/P1 elements}& & \multicolumn{4}{c}{P1/P2/P1 elements}\\
          $\kappa_\mathrm{f}$ & &
          $L^2S_{0,0,0}$ & 
          $L^2S_{0,0,1}$ & 
          $L^2S_{-0.5,0,1}$ &
          $L^2S_{-1,0,1}$ & &
          $L^2S_{0,0,0}$ & 
          $L^2S_{0,0,1}$ & 
          $L^2S_{-0.5,0,1}$ &
          $L^2S_{-1,0,1}$\\
          \cmidrule{1-1}\cmidrule{3-6}\cmidrule{8-11}
$10^{-7}$ && 10 & 8.18     & 8.18   & 8.36 && 10         & 6.36    & 6.27   & 6.64  \\
$10^{-8}$ && 12 & 9.91     & 9.82   & 9.91 &&11.91      & 9       & 8.45   & 8\\
$10^{-9}$ && 15.09 & 15.09 & 12.36  & 11.18  && 15.45      & 15.36   & 12.55  & 9.55  \\
$10^{-10}$ && 67.18 & 67.28  & 40   & 55 && 74.64     & 74.73   & 44.27  & 19.27\\
$10^{-11}$ && 347.55 & 348.18 & 194  & --&& 419.64    & 420.45  & 232    & -- \\
$10^{-12}$ && -- & --  & --   & -- && --        & -- & -- & --  \\\bottomrule
     \end{tabular}
          \caption{Average iteration count of the $L^2$--stabilized solvers for a varying $\permeability$ in the swelling test. Non-convergence denoted by -- (more than 500 iterations in this case).\label{tab:fp-s:permeability}}
  \end{table}

We note that this method is very sensitive to low permeabilities.
In previous studies, e.g.~\cite{Both2018}, Anderson acceleration has been shown to successfully increase robustness of stabilized iterative solvers. So we present the iteration counts for the same test but using Anderson acceleration with a depth of 5 in Table~\ref{tab:fp-s:permeability-accelerated}. We note that not only there is a significant decrease in the number of splitting iterations required (up to ca.\ 80\% for very low permeabilities), but it also enables the convergence of configurations which have previously not converge, again verifying previous observations. As long as the permeability is not too low, again aggressive $\displacement$ stabilization leads to the best performance. 
 
\begin{table}[!ht]
\footnotesize
\centering
     \begin{tabular}{ccccccccccc}\toprule
     && \multicolumn{4}{c}{P1/P1/P1 elements}& & \multicolumn{4}{c}{P1/P2/P1 elements}\\
          $\kappa_\mathrm{f}$ & &
          $L^2S_{0,0,0}$ & 
          $L^2S_{0,0,1}$ & 
          $L^2S_{-0.5,0,1}$ &
          $L^2S_{-1,0,1}$ & &
          $L^2S_{0,0,0}$ & 
          $L^2S_{0,0,1}$ & 
          $L^2S_{-0.5,0,1}$ &
          $L^2S_{-1,0,1}$\\
          \cmidrule{1-1}\cmidrule{3-6}\cmidrule{8-11}
$10^{-7}$ && 5.9     & {5.73}    & 6      & 6      && 5.73      & 4.91  & 4.91     & {4.91}  \\
$10^{-8}$ && {7}       & 7.27    & 7.27   & 7.09   && 6.91      & 6.91    & 6.64   & {5.91}  \\
$10^{-9}$ && 10.36   & 10      & 8.91   & {8.91}   && 10.45     & 10      & 9      & {7.09} \\
$10^{-10}$ && 18.91  & 18.09   & 14.91  & {12}     && 18        & 20.09   & 15.73  & {10}  \\
$10^{-11}$ && 43.55  & 45.18   & 33.73  & {26.18}  && 56.82     & 53.18   & 38.91  & {18.82} \\
$10^{-12}$ && {107.09} & 112.73  & 121.55  & --    && 140.73    & 117.36  & {95.64}  & 280.82\\
\bottomrule
     \end{tabular}
     \caption{Average iteration count of the $L^2$--stabilized solvers for a varying $\permeability$ in the swelling test using Anderson acceleration with depth 5. Non-convergence denoted by -- (more than 500 iterations in this case).}
     \label{tab:fp-s:permeability-accelerated}
 \end{table} 
 
\subsubsection{Dependence on densities} 
In Table~\ref{tab:fp-s:p1p1p1:density}, the iteration counts for varying $\rhos=\rhof$ are displayed. We observe that for very large densities the problem starts to become more difficult  to solve. To explain, increasing densities (merely) raise the second issue mentioned in Section~\ref{section:l2s-analysis-swelling-permeability-study}; in particular, as expected, destabilizing the solid equation does not yield any improvement, in contrast to the previous test. Iteration counts are identical for P1/P1/P1 and P1/P2/P1 elements. Thus, only the former is presented.

\begin{table}[!ht]
    \centering
     \begin{tabular}{cccccc}\toprule
     &&\multicolumn{4}{c}{
     \centering P1/P1/P1 elements}\\
          $\rhos=\rhof$ &&
          $L^2S_{0,0,0}$ & 
          $L^2S_{0,0,1}$ & 
          $L^2S_{-0.5,0,1}$ &
          $L^2S_{-1,0,1}$  \\
          \cmidrule{1-1}\cmidrule{3-6} 
$10^{2}$ && -- & 4.0 & 3.9 & 4.0 \\
$10^{4}$ && -- & 4.0 & 3.9 & 4.0 \\
$10^{6}$ && -- & 4.0 & 4.0 & 4.0 \\
$10^{8}$ && -- & 18.4 & 18.7 & 19.4 \\\bottomrule
     \end{tabular}
     \caption{Average iteration count of the $L^2$--stabilized solvers for a varying $\rhos=\rhof$ in the swelling test. Non-convergence denoted by --.}
     \label{tab:fp-s:p1p1p1:density}
 \end{table}

\subsubsection{Dependence on drained bulk modulus} 
In Table~\ref{tab:fp-s:bulk-modulus}, the iteration counts for varying $K_\mathrm{dr}$ (with same Poisson ratio) are displayed. We observe that lower drained bulk modulus is associated to higher iteration counts. 
This can be explained along the lines of the discussion of the dependence on the permeability, cf.\ Section~\ref{section:l2s-analysis-swelling-permeability-study}, since a lower drained bulk modulus leads to dominance of the $L^2$-type contribution of the $\displacement$-$\fluidvelocity$ block. Therefore, as expected, (aggressive) destabilization is beneficial. Additionally, a lower drained bulk modulus leads to a stronger coupling strength, and in accordance to Theorem~\ref{theorem:l2s-convergence}, to a deteriorating convergence rate. Again, inf-sup stability of the discretization of the fluid-pressure coupling enables slightly improved results, especially for low bulk modulus.

\begin{table}[!ht]
\footnotesize
\centering
     \begin{tabular}{ccccccccccc}\toprule
     && \multicolumn{4}{c}{P1/P1/P1 elements}& & \multicolumn{4}{c}{P1/P2/P1 elements}\\
          $K_\mathrm{dr}$ & &
          $L^2S_{0,0,0}$ & 
          $L^2S_{0,0,1}$ & 
          $L^2S_{-0.5,0,1}$ &
          $L^2S_{-1,0,1}$ & &
          $L^2S_{0,0,0}$ & 
          $L^2S_{0,0,1}$ & 
          $L^2S_{-0.5,0,1}$ &
          $L^2S_{-1,0,1}$\\
          \cmidrule{1-1}\cmidrule{3-6}\cmidrule{8-11}
          47.77 && --   & --   & --   & --    && --    & 26    & 19.18 & 16.64 \\
          477.7 && --   & 30   & 30   & 30    && --    & 10.82 & 9.91  & 9.91 \\
          4777  && 10   & 8.18 & 8.18 & 8.36  && 10.27 & 6.73  & 6.73  & 6.27 \\
         47770  && 5.82 & 4.82 & 4.91 & 5.18  && 6.73  & 5.73  & 5.73  & 5.64 \\ \bottomrule
        \end{tabular}
    \caption{Average iteration count of the $L^2$--stabilized solvers for a varying $K_\mathrm{dr}$ in the swelling test. Non-convergence denoted by --. \label{tab:fp-s:bulk-modulus}}
\end{table} 

%
%




\subsection{Comparison of the alternating minimization and \texorpdfstring{$L^2$}{L2}--stabilized splits}

The previous two sections allow for a first comparison of the two proposed schemes. In particular, two conclusions on the respective limitations can be made: (i) for increasing solid bulk modulus, the alternating minimization split quickly deteriorates, whereas the $L^2$-stabilized split remains robust; and (ii) for lower permeabilities, the performance of both schemes deteriorates, but the alternating minimization split in fact better handles the limit of very low permeabilities.

In this section, we continue the comparison of the two proposed schemes, now based on all the three suggested test cases with the parameters given in their description, enjoying different problem characteristics. 
The focus of the following study will also be to assess the impact of actual inf-sup stability, given for a Taylor-Hood like P2/P2/P1 discretization, opposed to the previously considered P1/P2/P1 discretization. Moreover, having observed the improving effect of Anderson acceleration in Section~\ref{section:l2s-analysis-swelling-permeability-study}, we follow this lead and also investigate the performance of the accelerated splits, this time also for the alternating-minimization. We also consider only the $L^2S_{-0.5,0,1}$ as it is the one suggested by the analysis and it exhibits an overall more robust performance.

For the swelling test, we additionally consider two bulk moduli, $\kappa_\mathrm{s}\in\{10^4, 10^8\}$. Results are presented in Table~\ref{table:iterative:swelling-comparison}. We observe that the inf-sup stability of the displacement plays no role, and the diagonally $L^2$--stabilized split proves very robust in all the tested scenarios, performing significantly better than the alternating minimization split. For the first, Anderson acceleration barely leads to improvement due to already low iteration counts; for the latter convergence can be significantly accelerated for the lower bulk modulus. For high bulk modulus, not even Anderson acceleration enables convergence.

\begin{table}[!ht]
\def\arraystretch{1}
    \centering
    \begin{tabular}{ccccccccccc}\toprule 
    && &&\multicolumn{3}{c}{P1/P2/P1}&&\multicolumn{3}{c}{P2/P2/P1}\\
    Method && $\kappa_\mathrm{s}$ &&  None & AA(1)  & AA(5) && None & AA(1)  & AA(5)\\
    \cmidrule{1-1}\cmidrule{3-3}\cmidrule{5-7}\cmidrule{9-11}
    Alt--min    && $10^4$ &&64.09 & 38.27   & 21  && 66.82 & 39.64  & 21.82\\
    $L^2S_{-0.5,0,1}$   && $10^4$    && 6.73  & 5.0    & 4.9 && 6.55  & 4.0   & 4.9 \\ 
    \cmidrule{1-1}\cmidrule{3-3}\cmidrule{5-7}\cmidrule{9-11}
    Alt--min    && $10^8$ && --    & --    & --   && --  & --   & --   \\
    $L^2S_{-0.5,0,1}$ && $10^8$      && 6.82  & 5.09    & 4.91 && 6.64  & 5.0   & 4.91 \\\bottomrule
    \end{tabular}
    \caption{Average iteration count for all tested scenarios in the swelling test, averaged over 5 time steps for $\kappa_\mathrm{s}\in\{10^4,10^8\}$. \textit{None} stands for the plain splits; AA($m$) stands for additional application of Anderson acceleration with depth $m$.}
\label{table:iterative:swelling-comparison}
\end{table}

We present the results of the footing test in Table \ref{table:iterative:footing-comparison}. We note that in this test the alternating minimization scheme exhibits lower iteration counts. Its success can be explained by the lower bulk modulus used, and instead the initial failure of the $L^2$--stabilized scheme is due to the permeability, which is very low. This case presents localized displacements at $\Gamma_\mathrm{foot}$, which are more affected by numerical locking, which justifies the increased iteration count in the case of the P2/P2/P1 discretization.

\begin{table}[!ht]
    \centering
    \begin{tabular}{ccccccccc}\toprule 
    &&\multicolumn{3}{c}{P1/P2/P1}&&\multicolumn{3}{c}{P2/P2/P1}\\
    Method && None & AA(1)  & AA(5) && None & AA(1)  & AA(5)\\
    \cmidrule{1-1}\cmidrule{3-5}\cmidrule{7-9}
    Alt--min    && 17.92 & 8.96  & 7.4  && 73.44 & 25.4   & 16.9  \\
    $L^2S_{-0.5,0,1}$         && --  & -- & 28.98  && --   & -- & 52.42 \\\bottomrule
    \end{tabular}
    \caption{Average iteration count for all tested scenarios in the footing test. \textit{None} stands for the plain splits; AA($m$) stands for additional application of Anderson acceleration with depth $m$.}
\label{table:iterative:footing-comparison}
\end{table}


The results of the perfusion test are presented in Table \ref{table:iterative:perfusion-comparison}. The behavior of this test is similar to the swelling one, with the $L^2$--stabilized split exhibiting a robust performance, which is further improved by the use of acceleration. The alternating minimization split instead presents difficulties in attaining convergence without acceleration, which can be explained by the use of a large bulk modulus. Similarly to the swelling test, the inf-sup stability of the displacement effectively plays no role.

\begin{table}[!ht]
    \centering
    \begin{tabular}{ccccccccc}\toprule 
    &&\multicolumn{3}{c}{P1/P2/P1}&&\multicolumn{3}{c}{P2/P2/P1}\\
    Method && None & AA(1)  & AA(5) && None & AA(1)  & AA(5)\\
    \cmidrule{1-1}\cmidrule{3-5}\cmidrule{7-9}
    Alt--min    && --   & 111.64 & 51.45  && --   & 134.36 & 50.18  \\
    $L^2S_{-0.5,0,1}$  && 18.36 & 10.27  & 9 && 14.64 & 9.36   & 8.09 \\\bottomrule
    \end{tabular}
    \caption{Average iteration count for all tested scenarios in the perfusion test. \textit{None} stands for the plain splits; AA($m$) stands for additional application of Anderson acceleration with depth $m$. \label{table:iterative:perfusion-comparison}}

\end{table}

%
%


\subsection{Comparison of splitting versus monolithic approaches}
In this section we present a comparison, in terms of computational time, between the proposed splitting schemes and a monolithic approach. 
We consider the swelling test and we choose the $L^2S_{-0.5,0,1}$ (labelled $L^2S$) as it yields the best performance for this problem.
The default stopping criterion for GMRES iterations is adopted for the monolithic scheme, with a relative tolerance equal to $10^{-8}$.
For the splitting scheme, the convergence tests for the linear system solved at each iteration is slightly relaxed, up to $10^{-6}$,
but the (relative) tolerance of the stopping criterion for the iterative splitting scheme is also set to $10^{-8}$, on the $\ell^\infty$ norm of the residual.
We compare the computational cost, measured by the average wall time per time step, calculated on a sequence of five consecutive time steps. 
Both formulations are solved using P1/P2/P1 finite elements, and the number of degrees of freedom is controlled by the number of nodes on each side of the domain.

The results of the comparison are reported in Table~\ref{tab:iterative-wall-time}.
The iterative schemes exhibit a better scaling with respect to the number of degrees of freedom.
In particular, for problems with over $10^5$ degrees of freedom (given by using 100 or more elements per side on the square domain) the wall time of the split scheme is consistently lower than of the monolithic approach. Also, the ratio between both solution times decreases monotonically with respect to the degrees of freedom as shown in the last column of the table, meaning that in this test case the superiority of iterative splitting schemes increases with the discrete problem size, which makes them a competitive solution strategy for addressing realistic scenarios, especially when considering tailored, possibly scalable preconditioners for the single subproblems.

\begin{table}[!ht]
    \centering
    \begin{tabular}{ccccccc}
        \toprule Nodes per side && \texttt{dofs} && $L^2S$ [s] & Monolithic [s] & ratio ($L^2S$ / Mono.)\\
        \cmidrule{1-1}\cmidrule{3-3} \cmidrule{5-7}
        50 && 28205   && 3.08   & 1.92 & 1.6042\\
        100 && 111405 && 11.62  & 15.61 & 0.7444\\
        150 && 249605 && 31.94  & 46.57 & 0.6858\\
        200 && 442805 && 61.79  & 128.49 & 0.4809\\
        250 && 691005 && 125.04 & 254.93 & 0.4905\\
        300 && 994205 && 196.97 & 569.36 & 0.3459\\ \bottomrule
    \end{tabular}
    \caption{Wall time [s] of the different approaches for increasing number of degrees of freedom.}
    \label{tab:iterative-wall-time}
\end{table}

\section{Discussion and Conclusions\label{section:conclusion}}
 
In this work we have developed splitting schemes for the linearized poromechanics problem studied in \cite{Burtschell201928,BARNAFI2020}, namely the alternating minimization split and the diagonally $L^2$--stabilized split. As the choice of a splitting scheme strongly depends on the application of interest, due to the strong dependence that the performance of each scheme has on the parameters, we tested the proposed methods on several benchmark problems.

The conclusions of this work arise from both theoretical analysis and numerical experiments. From the standpoint of theoretical convergence properties, we observe that the effectiveness of a splitting scheme hinges on the assumptions used for the convergence analysis and the corresponding stabilization, if necessary. For instance, the alternating minimization scheme requires the algebraic inversion of the pressure, so it can be expected for it to deteriorate whenever this operation is not admissible ($(1-\phi)/\kappa_\mathrm{s}\rightarrow \infty$). The diagonally $L^2$--stabilized split can be interpreted as an approximate Schur complement method, where the $L^2$-type contributions are not considered. This implies that it can be expected for such $L^2$--stabilized schemes to present difficulties converging whenever the $L^2$-type contributions are dominant, meaning small permeability or large densities. The analysis also provides the interesting possibility of destabilizing the solid momentum equation in the diagonally $L^2$--stabilized scheme. 

Such trends are confirmed by numerical experiments. The alternating minimization scheme performs very well in compressible scenarios but its convergence rate quickly deteriorates as the bulk modulus increases. The diagonally $L^2$--stabilized split instead is robust with respect to the bulk modulus, so it should be preferred in (quasi-)incompressible regimes. The numerical experiments also confirm that the destabilization of the solid momentum equation yields good improvements of the convergence rate.
Neither of the schemes is capable of handling large densities or small permeabilities -- enhanced splitting schemes which successfully incorporate the $L^2$-type contribution in the \mbox{displacement--fluid} velocity block are a topic of future research. Still, an improvement for the low permeability scenario can be seen by using inf-sup stable elements for the fluid-pressure block. This is an interesting property to be investigated, as it does not emerge in the analysis.

We have strengthened our splitting schemes with Anderson acceleration, which is a general method to improve the convergence of fixed-point iterations. It does not only improve the convergence of all methods tested, but it also enables convergence in scenarios in which it previously would not converge. Another feature of Anderson acceleration, particularly relevant in this framework, is that it reduces the influence of the stabilization parameters. This is indeed a fundamental aspect, as the user-defined choice and tuning of parameters represent a drawback of the presented methods. 

Finally, we have compared the diagonally $L^2$--stabilized split with a monolithic approach applied to the linearized problem. This study shows that for a sufficiently large size of the discrete problem, the iterative splitting approach is a competitive choice. Such methods may then be rightfully considered as effective options for solving realistic poromechanics problems applied to soft materials. Further investigations considering practical biomedical applications will be performed in the future.

\section*{Acknowledgements}
PZ has been supported by the Italian research project MIUR PRIN17 2017AXL54F “Modeling the heart across the scales: from cardiac cells to the whole organ”;
NB and AQ received funding from the European Research Council (ERC) under the European Union’s Horizon 2020 research and innovation programme (grant agreement No 740132); JWB and FAR have in part been supported by the Research Council of Norway (RCN) Project 250223; in addition, JWB has in part been supported by the FracFlow project funded by Equinor through Akademiaavtalen.

\begin{center}
\raisebox{-.5\height}{\includegraphics[width=.15\textwidth]{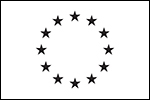}}
\hspace{2cm}
\raisebox{-.5\height}{\includegraphics[width=.15\textwidth]{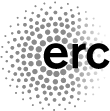}}
\end{center}


\end{document}